\documentclass[a4paper,11pt,draft]{amsart}
\usepackage[margin=25mm]{geometry}
\usepackage{amssymb, amsmath, amsthm, amscd}

\usepackage{amsmath}
\allowdisplaybreaks[4]

\usepackage[T1]{fontenc}
\usepackage{eucal,mathrsfs,dsfont}
\usepackage{color}

\newcommand{\normal}{\color{black}}

\usepackage{marginnote}

\definecolor{mno}{rgb}{0.5,0.1,0.5}

\newcommand{\R}{\mathds R}
\newcommand{\N}{\mathds N}

\newcommand{\Pp}{\mathds P}
\newcommand{\Ee}{\mathds E}

\newcommand{\I}{\mathds 1}

\def\<{\langle}
\def\>{\rangle}

\newcommand{\casymp}[1]{\stackrel{#1}{\asymp}}

\newtheorem{theorem}{Theorem}[section]
\newtheorem{lemma}[theorem]{Lemma}
\newtheorem{proposition}[theorem]{Proposition}
\newtheorem{corollary}[theorem]{Corollary}

\theoremstyle{definition}

\newtheorem{example}[theorem]{Example}
\newtheorem{remark}[theorem]{Remark}

\begin{document}
\allowdisplaybreaks
\title[Fractional Schr\"{o}dinger Operators with Confining Potentials] {\bfseries Heat kernels and Green functions for fractional Schr\"{o}dinger operators with confining potentials}

\author{Xin Chen,\quad  Kamil Kaleta,\quad
Jian Wang}

\thanks{\emph{X.\ Chen:}
   Department of Mathematics, Shanghai Jiao Tong University, 200240 Shanghai, P.R. China. \texttt{chenxin217@sjtu.edu.cn}}
   \thanks{\emph{K.\ Kaleta:}
Wroc{\l}aw University of Science and Technology,  Wybrze\.{z}e St. Wyspia\'{n}skiego 27, 50-370 Wroc{\l}aw, Poland.
\texttt{kamil.kaleta@pwr.edu.pl}}
   \thanks{\emph{J.\ Wang:}
    School of Mathematics and Statistics \& Key Laboratory of Analytical Mathematics and Applications (Ministry of Education) \& Fujian Provincial Key Laboratory
of Statistics and Artificial Intelligence, Fujian Normal University, 350007 Fuzhou, P.R. China. \texttt{jianwang@fjnu.edu.cn}}

\date{}

\maketitle

\begin{abstract}
We give two-sided, global (in all variables) estimates of the heat kernel and the Green function of the fractional Schr\"odinger operator with a non-negative and locally bounded potential $V$ such that $V(x) \to \infty$ as $|x| \to \infty$. We assume that $V$ is comparable to a radial profile with the doubling property. Our bounds are sharp with respect to spatial variables and qualitatively sharp with respect to time. The methods we use combine probabilistic and analytic arguments. They are based on
the strong Markov property and
the Feynman--Kac formula.
\medskip

\medskip

\end{abstract}
\allowdisplaybreaks

\section{Introduction and main results}
Let $d \in \N :=\left\{1,2,\dots \right\}$ and $\alpha \in (0,2)$. Let $\mathcal L :=-(-\Delta)^{\alpha/2}$ be the \emph{fractional Laplacian} of order $\alpha$ -- the non-local, pseudo-differential operator which is defined by
$$
\mathcal Lf(x)=C_{d,\alpha}\int_{\R^d}\left(f(x+z)-f(x)-\langle\nabla f(x),z\rangle\rangle\I_{\{|z|\le 1\}}\right)\frac{1}{|z|^{d+\alpha}}\,dz,\quad f\in C_c^{\infty}(\R^d)
$$
with
\[
C_{d,\alpha} = \frac{ \alpha 2^{\alpha-1}\Gamma\big((d+\alpha)/2\big)}{\pi^{d/2}\Gamma(1-\alpha/2)}.
\]
It is well known that $\mathcal L$ is an infinitesimal generator of the rotationally invariant $\alpha$-stable L\'evy process $X:=\{X_t\}_{t\ge 0}$ in $\R^d$, see e.g.\ \cite{K} for various definitions of $\mathcal L$ and the probabilistic background.

The main goal of the paper is to find global, two-sided estimates of the heat kernel and the Green function
associated with
the Schr\"odinger operator
\begin{equation}\label{e1-1}
\mathcal L^V=\mathcal L-V.
\end{equation}
Here and in what follows, the potential $V:\R^d\to [0,\infty)$ is a sufficiently regular, non-negative and locally bounded function such that $V(x) \to \infty$ as $|x| \to \infty$. Therefore, $\mathcal L^V$ extends to a (negative) self-adjoint operator on
$L^2(\R^d):=L^2(\R^d,dx)$, see e.g. \cite[Proposition 10.22]{Sch}.

The operator $-\mathcal L^V$ can be interpreted as a Hamiltonian in the fractional version of the classical model of \emph{oscillator}, where $\mathcal L^V$ takes the form $\Delta-V$, with $\Delta$ denoting the standard Laplacian and $V$ being a potential as defined above. The most classical case, $V(x) = c |x|^2$ (where $c>0$), corresponds to the harmonic oscillator. However, potentials exhibiting different growth rates at infinity are also of interest, such as the anharmonic (quartic) \cite{BW} or logarithmic \cite{CZTH} oscillators. In these models, particles are confined by forces corresponding to potentials $V$ that grow to infinity at infinity. Such potentials are, therefore, commonly referred to as \emph{confining} potentials in the literature. Fractional operators, especially in the \emph{relativistic} case with $\alpha=1$, play a significant role in contemporary mathematical physics (see, e.g.,\ \cite{FMS,LS}).

The Schr\"odinger semigroup corresponding to the operator $\mathcal L^V$ is denoted by $\{T_t^V\}_{t\ge 0}$; it fulfills the Feynman--Kac formula
\begin{equation}\label{e1-4}
T_t^V f(x)=\Ee_x\left[\exp\left(-\int_0^t V(X_s)\,ds\right)f(X_t)\right],\quad f \in L^2(\R^d) \text{\ or } f \in C_b(\R^d)
\end{equation}
($\Pp_x$ and $\Ee_x$ denote the probability measure and the expectation with respect to starting position $x\in \R^d$, respectively). Moreover, there exists a jointly continuous and symmetric integral kernel $p:(0,\infty) \times \R^d\times \R^d \to (0,\infty)$ such that
\begin{equation}\label{e1-5}
T_t^V f(x)=\int_{\R^d}p(t,x,y)f(y)\,dy,\quad t>0, \ f \in L^2(\R^d) \text{\ or } f \in C_b(\R^d),
\end{equation}
see e.g.\ the monograph by Demuth and van Casteren \cite[Theorem 2.5]{DC} for a reference. Since $V\ge0$,
$$p(t,x,y)\le q(t,x,y),\quad t>0,\ x,y\in \R^d,$$
where $q(t,x,y)$ is the heat kernel
associated with
the fractional Laplacian $\mathcal L$ (or, equivalently, the transition density of
the $\alpha$-stable process $X$). We remark that there are constants $c_1,c_2>0$ such that
\begin{equation}\label{e:heat1}
c_1\left(t^{-d/\alpha}\wedge\frac{t}{|x-y|^{d+\alpha}}\right)\le q(t,x,y)\le c_2\left(t^{-d/\alpha}
\wedge \frac{t}{|x-y|^{d+\alpha}}\right),\quad t>0,\ x,y\in \R^d,
\end{equation}
see \cite[Theorem 2.1]{BG} and \cite[Proof of Lemma 5]{BJ07}.

\medskip

Throughout the paper we use the following assumption
on the regularity and growth of $V$.

\medskip

\noindent
\textbf{Assumption (H)}
\begin{itemize}
\item[] {\it
There exist a non-decreasing function $g:[0,\infty)\to [1,\infty)$, satisfying $\lim_{r \to \infty}g(r)=\infty$,
and constants $C_1,C_2,C_3>0$
such that
\begin{equation}\label{e1-2}
g(2r)\le C_1 g(r),\quad r\ge0,
\end{equation}
and
\begin{equation}\label{e1-3}
C_2g(|x|)\le V(x)\le C_3g(|x|),\quad |x|\ge 1.
\end{equation}}
\end{itemize}

\medskip

Though the theory of non-local Schr\"odinger operators with confining potentials and their semigroups is an active area of research, existing results are mainly focused on
intrinsic contractivity properties and ground state estimates.
For basic definitions, we refer the reader to Davies and Simon \cite{DS84}. A sharp necessary and sufficient condition for the intrinsic ultracontractivity of the non-local Schr\"odinger operator $\mathcal L^V$, driven by relativistic and usual rotationally invariant $\alpha$-stable processes, was provided by Kulczycki and Siudeja \cite{KS}, and by Kaleta and Kulczycki \cite{KK10}, respectively. Later, Kaleta and L\H{o}rinczi \cite{KL} extended these results to non-local Schr\"odinger operators $\mathcal L^V$ associated with the class of symmetric L\'evy processes possessing the direct jump property (DJP). The authors developed a new iteration technique to establish eigenfunction estimates for $\mathcal L^V$, including two-sided sharp bounds for the ground state, and characterized the asymptotic version (valid for large times) of intrinsic ultracontractivity.
Chen and Wang \cite{CW} applied the intrinsic super-Poincar\'e inequality to establish a sharp criterion for intrinsic ultracontractivity and derived two-sided estimates for the ground state of non-local Schr\"odinger operators associated with general symmetric jump processes. Their analysis accounted for both finite-range and long-range jump behaviors. Furthermore, by employing a similar approach, the authors investigated various hyperboundedness properties of
the
associated Schr\"odinger semigroups \cite{CW1}.
More recently, Kaleta and Schilling \cite{KS} further analyzed the DJP class, providing large-time heat kernel estimates for the corresponding non-local Schr\"odinger operators with confining potentials. These estimates were established for $t\ge T_0$, with $T_0>0$ properly chosen.
However, to the best of our knowledge, \emph{global estimates of the heat kernel $p(t,x,y)$ for the full time regime,
as well as estimates for the Green function $G(x,y)$ associated with $\mathcal L^V$,
remain unknown.} In this paper, we address this problem comprehensively. We focus on the fractional Schr\"odinger operator $\mathcal L^V$ and prove two-sided estimates for $p(t,x,y)$ over the full time horizon,
as well as two-sided estimates for the Green function $G(x,y)$.

\ \

To state our first main result, we need to introduce some more notation. Define
\begin{equation}\label{e1-6}
t_0(s):=\inf\left\{t>0: \exp\left(-tg(s)\right)=\frac{t}{(1+s)^{\alpha}}\right\},\quad s\ge0.
\end{equation}
Note that, for any fixed $s\ge0$, the function $t\mapsto e^{-t g(s)}$ is strictly decreasing and onto the interval $(0,1)$, while $t\mapsto\frac{t}{(1+s)^{\alpha}}$ is strictly increasing, onto $(0,\infty)$. Then, $t_0(s)$ is unique in the sense that for any $t<t_0(s)$, $\exp\left(-tg(s)\right)>\frac{t}{(1+s)^{\alpha}}$, and for $t>t_0(s)$,  $\exp\left(-tg(s)\right)<\frac{t}{(1+s)^{\alpha}}$.
See Lemma \ref{dbl-polyn} below for further estimates of $t_0(s)$.

Also, we say that the function $t_0:[0,\infty)\to (0,\infty)$ is almost increasing  (resp.\ decreasing), if there exists a continuous,
strictly increasing (resp.\ strictly decreasing) function $h:[0,\infty) \to (0,\infty)$ and constants $C_*,C_{**}\ge 1$ such that
\begin{equation}\label{e4-1}
C_*^{-1}h(s)\le t_0(s)\le C_*h(s),\quad C_{**}^{-1}h(s)\le h(2s)\le C_{**}h(s),\quad s\ge0.
\end{equation}
For any $a,b\ge0$, $a\wedge b:=\min\{a,b\}$ and $a\vee b:=\max\{a,b\}.$

Under Assumption \textbf{(H)} it holds that $\lim_{|x|\to \infty}V(x)=\infty$. Then, according to
\cite[Corollary 1.3]{WW} (see also \cite[Proposition 1.1]{CW} or \cite[Lemma 1]{KK10}), the  operator $T_t^V: L^2(\R^d)\to L^2(\R^d)$ is compact for every $t>0$. Hence, by general
theory of semigroups of compact operators, there exists an orthonormal basis in
$L^2(\R^d)$, consisting of eigenfunctions $\{\phi_n\}_{n\ge 1}$ of the operator $-\mathcal L^V$, with associated eigenvalues $\{\lambda_n\}_{n\ge 1}$ satisfying
$0<\lambda_1<\lambda_2\le \lambda_3\le \cdots$ and $\lim_{n \to \infty}\lambda_n=\infty$. We usually call
$\phi_1$ the ground state of $-\mathcal L^V$.
In particular, for all $t>0$ and $x\in \R^d$,
\begin{equation}\label{l4-6-1}
T_t^V \phi_1(x)=e^{-\lambda_1 t}\phi_1(x).
\end{equation}
Furthermore, it is known that $\phi_1:\R^d \to \R_+$ has a version which is continuous and strictly positive (see e.g. \cite[Proposition 1.2]{CW}). Then, by
\cite[Corollary 3]{KK10}, for all $x\in \R^d$,
$$
c_1H(x)\le \phi_1(x)\le c_2H(x),
$$ where $$H(x):=\frac{1}{g(|x|)(1+|x|)^{d+\alpha}}.$$
This, along with \eqref{l4-6-1}, yields that
\begin{equation}\label{l4-6-2}
c_3e^{-\lambda_1 t}H(x) \le T_t^V H(x)\le c_4e^{-\lambda_1 t}H(x),\quad  x\in \R^d,\ t>0.
\end{equation}
The following theorem is the first main result of the paper.

\begin{theorem} \label{th:main}
Let Assumption {\bf (H)} hold. Then there are constants $C_0,\tilde C_0, C_1\ldots,C_{12}>0$ such that
for every $x,y\in \R^d$ and $t>0$ the following statements hold.
\begin{itemize}
\item [(1)]If $0<t\le C_0t_0(|x|\wedge |y|)$, then
\begin{equation}\label{t1-1-1}
\begin{split}
 C_{1}\left(t^{-d/\alpha}\wedge \frac{t}{|x-y|^{d+\alpha}}\right) & \left(\frac{1}{tg(|x|\vee|y|)}\wedge1\right) e^{-C_{2}tg(|x|\wedge|y|)} \le p(t,x,y) \\
& \le C_{3}\left(t^{-d/\alpha}\wedge \frac{t}{|x-y|^{d+\alpha}}\right)\left(\frac{1}{tg(|x|\vee|y|)}\wedge1\right) e^{-C_{4}tg(|x|\wedge |y|)}.
\end{split}
\end{equation}
\medskip
\item[{\rm (2-i)}] If $t_0(\cdot)$ is almost increasing, then, for $t>C_0t_0(|x|\wedge |y|)$,
\begin{equation}\label{t1-1-2}
\begin{split}
  C_{5}e^{-\lambda_1t}H(x)H(y)\le p(t,x,y)
\le C_{6}e^{-\lambda_1t}H(x)H(y).
\end{split}
\end{equation}

\medskip
\item[{\rm (2-ii)}] If $t_0(\cdot)$ is almost decreasing, then,
for
$C_0t_0(|x|\wedge |y|)<t \le \tilde C_0$,
\begin{equation}\label{t1-1-3}
\begin{split}
 C_7H(x)H(y)\int_{\{|z|\le s_0(t)\}}e^{-C_{6}tg(|z|)}\,dz & \le p(t,x,y)\\
 &\le C_8H(x)H(y)\int_{ \{|z|\le s_0(C_{9}t)\}}e^{-C_{10}tg(|z|)}\,dz . \end{split}
\end{equation}
where $s_0(t)=h^{-1}(t)\vee 2$, $
h^{-1}(t):=\inf\{s\ge0: h(s)\le t\}
$
with $h$ given in \eqref{e4-1}, and we use the convention that $\inf \emptyset=-\infty$;
for every $t>\tilde C_0$,
\begin{equation}\label{t1-1-4}
C_{11}e^{-\lambda_1t}H(x)H(y)\le p(t,x,y)
\le C_{12}e^{-\lambda_1t}H(x)H(y).
\end{equation}
\end{itemize}

\end{theorem}

\bigskip

We now illustrate our result with the following example.

\begin{example} \label{ex1} \it Let
\[
V(x)=\log ^\beta (1+|x|), \qquad \beta>0.
\]
Then there are constants $C_0,\ldots,C_{10}>0$ such that for every $x,y\in \R^d$ with $|x|\le |y|$ and $t>0$ we have the following estimates.
\begin{itemize}
\item[{\rm (1)}] If $0<t\le C_0\log^{1-\beta}(2+|x|)$, then
\begin{align*}
C_1 & \left(t^{-d/\alpha}\wedge \frac{t}{|x-y|^{d+\alpha}}\right) \left(\frac{1}{t
\log^\beta(2+|y|)
 }\wedge1\right) e^{-C_2 t
 \log^\beta(2+|x|)
 } \le p(t,x,y) \\
& \ \ \ \ \ \ \ \ \ \
    \le C_3 \left(t^{-d/\alpha}\wedge \frac{t}{|x-y|^{d+\alpha}}\right) \left(\frac{1}{t\log ^\beta (2+|y|) }\wedge1\right) e^{-C_4 t \log ^\beta (2+|x|)}.
\end{align*}

\medskip

\item[{\rm (2-i)}] If
$t\ge C_0\log^{1-\beta}(2+|x|)$
and $\beta \le 1$, then
\begin{align*}
C_{5}
e^{-\lambda_1 t}
H(x) H(y) \le  p(t,x,y) \le C_{6} e^{-\lambda_1 t} H(x) H(y).
\end{align*} Here and in what follows,
\[
H(x) = \frac{1}{\log ^\beta (2+|x|) (1+|x|)^{d+\alpha}},\quad x\in \R^d.
\]

\medskip

\item[{\rm (2-ii)}]
If $t\ge C_0\log^{1-\beta}(2+|x|)$
and $\beta>1$, then
\begin{align*}
C_7  \left(e^{C_{8}
t^{-1/(\beta-1)}} \I_{\{t \le 1\}}+ e^{-\lambda_{1} t}\I_{\{t > 1\}}\right) & H(x) H(y) \le p(t,x,y) \\
& \le C_{9}  \left(e^{C_{10}
t^{-1/(\beta-1)}}\I_{\{t \le 1\}}+ e^{-\lambda_{1} t}\I_{\{t >1\}}\right)H(x) H(y) .
\end{align*}
\end{itemize}
\end{example}

\ \

Our second main result gives estimates for the Green  function $G^V(x,y):= \int_0^{\infty} p(t,x,y) \,dt$.
In what follows the notation $f(x) \asymp g(x)$, $x \in A$, means that there is a constant $c \geq 1$ such that $(1/c)g(x) \le f(x) \le c g(x)$, $x \in A$ (we also write ``$\casymp{c}$'' to indicate the comparison constant $c$).

\begin{theorem}\label{prop:Green}
Let
Assumption {\bf (H)} hold.  If $t_0(\cdot)$ is an almost monotone function, then
\[
G^V(x,y) \asymp  \frac{1}{|x-y|^{d-\alpha}} \left(1 \wedge \frac{1}{g(|x|)|x-y|^{\alpha}}\right) \left(1 \wedge \frac{1}{g(|y|)|x-y|^{\alpha}}\right) Q(x,y), \quad x,y \in \R^d, \
x\neq y.
\]
where
\[
Q(x,y):=
\begin{cases}
1,\ &\alpha < d,\\
\max \left\{\log \left(\frac{1}{|x-y|^\alpha g(|x| \wedge |y|)}\right),1\right\},\ & \alpha = d =1, \\
 \max \left\{\left(\frac{1}{|x-y|^{\alpha}g(|x| \wedge |y|)}\right)^{\frac{\alpha-d}{\alpha}},1\right\},\ & \alpha > d =1.
\end{cases}
\]
\end{theorem}

Next, we will make some remarks on our main results above.
\begin{itemize}
\item[(1)]
In order to make the presentation and proofs in this paper as clear and readable as possible, we decided to focus on the case when $X:=\{X_t\}_{t\ge 0}$ is the rotationally invariant $\alpha$-stable process. This choice also makes our contribution accessible to a wider audience. However, we believe that the arguments we propose apply to a large class of symmetric jump processes.
In fact, it is straightforward to observe that the expression \eqref{l2-4-2} for the heat kernel $p(t,x,y)$ relies solely on the strong Markov property of $X$ and the Feynman-Kac formula \eqref{e1-4}. Using this, along with the L\'evy system \eqref{l2-1-1} and Dirichlet heat kernel estimates for more general symmetric jump processes established in \cite{GKK, KK}, it is possible to apply suitable iteration arguments (by choosing an appropriate domain $U$ when applying \eqref{l2-4-2}) to derive two-sided estimates for the Schr\"odinger heat kernel $p(t,x,y)$ associated with such processes.

\item [(2)] As mentioned above, the expression \eqref{l2-4-2} serves as the starting point for our proofs. It allows us to derive appropriate estimates for the kernel $p(t,x,y)$ by investigating the probabilities that the process $X$ visits various positions where the potential $V$ takes specific values. In contrast to \cite{CW1}, where the classical Schr\"odinger semigroup was studied, here we apply the L\'evy system formula \eqref{l2-1-1} to estimate the joint distribution $\left(\tau_U, X_{\tau_U}\right)$ in \eqref{l2-4-2}. Our approach here is different because our process $X$ is a jump process.  We also note that our argument in this paper differs significantly from that in \cite{KS}, where the DJP and the estimates for harmonic functions played a crucial role.

\item[(3)] According to Theorem \ref{th:main}, we know that there exists a threshold time
 $C_0t_0(|x|\wedge |y|)$ for the two-sided estimates of the heat kernel $p(t,x,y)$. Intuitively, if $t\le C_0t_0(|x|\wedge |y|)$, the typical path from $x$ to $y$ that maximizes the probability for the process $X$ evolving in the presence of the potential $V$ follows the path of the free process $X$ (i.e.,\ when $V=0$), and the effect of the potential along this path is expressed only through the local values of $V$ near $x$ and $y$. When $t>C_0t_0(|x|\wedge |y|)$, the corresponding path that maximizes the probability will first travel from $x$ to a domain near the origin, where the effect of the potential $V$ is small, and then proceed from this domain to $y$, where the effect of the potential $V$ is more significant. For example, see the integral terms in \eqref{t1-1-3}.
\end{itemize}

\ \

The rest of the paper is organized as follows. In the next section, we present some necessary preliminaries for the proof of Theorem \ref{th:main}. In particular, we establish an upper bound for $T_{t}^V 1(x)$ and some rough off-diagonal estimate of $p(t,x,y)$. Section \ref{section3} is devoted to two-sided estimates of $p(t,x,y)$ when $0<t\le C_0't_0(|x|\wedge |y|)$ for any $C_0'>0$.  In Section \ref{section4}, we verify two-sided estimates of $p(t,x,y)$ when $t\ge C_0t_0(|x|\wedge |y|)$
for some $C_0>0$. In Section \ref{section5}, we provide the proof of Example \ref{ex1}.  Finally, in Section \ref{section6}, we present the proofs of Theorems \ref{th:main} and \ref{prop:Green}, based on our findings from Sections \ref{section3} and \ref{section4}.

\section{Preliminaries}

The following identity, called the L\'evy system formula, describes the structure of jumps of the (rotationally invariant)
$\alpha$-stable process $X:=\{X_t\}_{t\ge 0}$, see e.g.\ \cite[Lemma 4.7]{CKu}.

\begin{lemma}\label{l2-1}
Suppose that $f:\R_+\times \R^d\times \R^d\to [0,\infty)$ is a nonnegative measurable  function so that
$f(s,x,x)=0$ for all $x\in \R^d$ and $s\in \R_+$. Then, for every $x\in \R^d$ and stopping time $\tau$, it holds that
\begin{equation}\label{l2-1-1}
\Ee_x\left[\sum_{s\le \tau}f (s,X_{s-},X_s )\right]=\Ee_x\left[\int_0^\tau \int_{\R^d}
f (s,X_{s-},z )\frac{C_{d,\alpha}}{|X_{s-}-z|^{d+\alpha}}\,dz \,ds\right],
\end{equation}
where $X_{s-}:=\lim\limits_{t\to s \atop t<s} X_t$ denotes the left limit of $X_{\cdot}$ at $s$.
\end{lemma}
For any open subset $D\subset \R^d$ and $x\in D$, we define by
$$
\delta_D(x):=\inf\{|x-z|: z\in
D^c\}
$$
the distance from $x\in D$ to the boundary $\partial D$.
Let $\tau_D:=\inf\{t\ge 0: X_t\notin D\}$ be the first exit time of the process
$X$ from the set $D$. Denote by $q_D(t,x,y)$ the transition density of $X$ killed upon exiting $D$.
The following two-sided estimates for $q_{B(x,R)}(t,x,y)$ essentially have been proved in \cite{CKS1}.
\begin{lemma}\label{l2-2}
For any constant $C_0'>0$ there exists a  positive constant $C_1$  such that for every $x\in \R^d$, $R>0$, $y\in B(x,R)$ and $0<t\le C_0'R^\alpha$,
\begin{equation}\label{e:heat2}
\begin{split}
q_{B(x,R)}(t,x,y)  \casymp{C_1}  \left(t^{-d/\alpha}\wedge\frac{t}{|x-y|^{d+\alpha}}\right)\left(\frac{\delta_{B(x,R)}(y)^{\alpha/2}}{\sqrt{t}}\wedge 1\right).
\end{split}
\end{equation}
\end{lemma}
\begin{proof}
Due to the spatial  homogeneity  of
the process $X$, it suffices to prove \eqref{e:heat2} for $x=0$.
According to \cite[Theorem 1.1]{CKS1}, for every $C_0'>0$ there exists a positive constant $c_1$ such that for all $ 0<t\le C_0'$ and $y\in B(0,1)$,
\begin{equation}\label{l2-2-1}
\begin{split}
q_{B(0,1)}(t,0,y)  \casymp{c_1} \left(t^{-d/\alpha}\wedge\frac{t}{|y|^{d+\alpha}}\right)\left(\frac{\delta_{B(0,1)}(y)^{\alpha/2}}{\sqrt{t}}\wedge 1\right).
\end{split}
\end{equation}
Furthermore, we note that $\delta_{B(0,1)}(y)=1-|y|$ for all $y\in B(0,1)$.
Combining this  and  \eqref{l2-2-1} with the scaling property
\begin{align*}
q_{B(0,1)}(t,0,y)=R^dq_{B(0,R)}(R^\alpha t,0,Ry),\quad t>0,\ y\in B(0,1),
\end{align*} we obtain \eqref{e:heat2} immediately.
\end{proof}

 We also give a lemma which describes  the properties of $t_0(\cdot)$ and $g(\cdot)$.

\begin{lemma}\label{dbl-polyn}
Let $t_0(\cdot)$ be defined by \eqref{e1-6} and assume that $g(\cdot)$ satisfies \eqref{e1-2} and \eqref{e1-3}.
Then there are constants $C_2, q_1 >0$ such that for all $s\ge0$,
\begin{equation}\label{l4-2-1}
g(s)\le C_2(1+s)^{q_1},
\end{equation} and
\begin{equation}\label{l4-2-1b}
C_2^{-1}(1+s)^{-q_1}\le t_0(s)\le C_2\left(1+\frac{\log(2+s)}{g(s)}\right).
\end{equation}
\end{lemma}

\begin{proof}
If $s > 1$, then $s \in (2^n,2^{n+1}]$ for some $n \in \left\{0,1,2,\ldots\right\}$. By the monotonicity of $g$ and \eqref{e1-2} (we may assume that  $c_0 >1$, where $c_0$ denotes the comparability constant in \eqref{e1-2}),
\[
g(s) \le g(2^{n+1}) \le c_0^{n+1} g(1) = c_0 g(1) 2^{n\log_2 c_0} \le c_1 s^{q_1} \le c_1 (1+s)^{q_1}
\]
with $c_1:=c_0 g(1)$ and $q_1:=\log_2 c_0$. If $s \in [0,1]$, then $g(s) \le g(1) \le c_1 \le c_1 (1+s)^{q_1}$.
By all of the above estimates we obtain \eqref{l4-2-1}.

By
\eqref{e1-6} it holds that $\exp(-t_0(s)g(s))=t_0(s)(1+s)^{-\alpha}$, so  $t_0(s)g(s)+\log t_0(s)=\alpha \log(1+s)$. If $t_0(s)\le c_2\log 2$ holds for some $c_2>0$ that satisfies $c_2\log 2\ge 1$,
then the second inequality in \eqref{l4-2-1b} follows.
If $t_0(s)\ge c_2\log 2(\ge1)$, then
$t_0(s)=\frac{\alpha \log (1+s)-\log t_0(s)}{g(s)}\le \frac{\alpha \log(1+s)}{g(s)}$
for $s>0$. Putting both estimates above together, we prove the second inequality in \eqref{l4-2-1b}.

By the fact that $g(r)\ge1$ for all $r\ge0$, there is $a_0>0$ large enough such that for all $s>0$,
\begin{align}\label{l4-2-1c}
(a_0g(s))^{-1}(1+s)^{-\alpha}\le e^{-a_0^{-1}}=\exp\left(-(a_0g(s))^{-1}g(s)\right).
\end{align}
By \eqref{e1-6} again, $t_0(s)\ge (a_0g(s))^{-1}\ge c_3(1+s)^{-q_1}$ for every $s>0$.
Then, we prove the first inequality  in \eqref{l4-2-1b}.
\end{proof}

Next, we present upper bounds for $T_t^V\I(x)$, which will be frequently used in the paper.

\begin{lemma}\label{L:1.1}
There exist positive constants  $C_3, C_4$  such that for any $t>0$ and $x\in \R^d$,
\begin{equation}\label{est1}
T_t^V 1(x)\le C_3 \left(\exp\left(-C_4tg(|x|)\right)+ \frac{t}{(1+|x|)^{\alpha}}\right).
 \end{equation}
 \end{lemma}

\begin{proof}
Since
\begin{align*}
\inf_{x\in B(0,2),t>0}\left(\exp\left(-C_4tg(|x|)\right)+ \frac{t}{(1+|x|)^{\alpha}}\right)\ge c_0,
\end{align*}
\eqref{est1} holds trivially for every $x\in \R^d$ with $|x|\le 2$ by taking $C_3$ large enough.
Then, it suffices to prove \eqref{est1} for every $x\in \R^d$ with $|x|\ge 2$. Set $U:=B(x,|x|/2)$.
According to \eqref{e1-4},
\begin{align*}
  T^V_t1(x)
&=\Ee_x\left[\exp\left(-\int_0^t V(X_s)\,ds\right)\right]\\
&=\Ee_x\left[\exp\left(-\int_0^t V(X_s)\,ds\right) \I_{\{\tau_U\ge t\}}\right]+ \Ee_x\left[\exp\left(-\int_0^t V(X_s)\,ds\right) \I_{\{\tau_U< t\}}\right]\\
&=:I_1+I_2.\end{align*}

Obviously, by \eqref{e1-2} and \eqref{e1-3},
$$I_1\le \exp\left(-t \inf_{z\in U} V(z)\right)\le \exp\left(-c_1tg(|x|)\right).$$

On the other hand,
\begin{align*}
I_2 &= \Ee_x \left[\exp\left(-\int_0^t V(X_s)\,ds\right)\I_{\{\tau_U<t,\, X_t\in B(x,|x|/3)\}}\right] \\
 & \ \ \ \ + \Ee_x\left[\exp\left(-\int_0^t V(X_s)\,ds\right)\I_{\{ X_t\notin B(x,|x|/3)\}}\right]\\
&=:I_{21}+I_{22}.
\end{align*}
By the strong Markov property,
\begin{align*}I_{21}&\le \Ee_x\left[\Pp_{X_{\tau_U}}(X_{t-\tau_U}\in B(x,|x|/3))\I_{\{\tau_U<t\}}\right] \le \sup_{ 0\le s\le t \atop z\in B(x,|x|/2)^c}\int_{B(x,|x|/3)}q(s,z,y)\,dy\\
&\le c_2\sup_{ z\in B(x,|x|/2)^c}\int_{B(x,|x|/3)} t|z-y|^{-d-\alpha}\,dy \le \frac{c_3t}{(1+|x|)^{\alpha}}, \end{align*}
where in the third inequality we used \eqref{e:heat1}.   Similarly, by using \eqref{e:heat1} once again,
$$I_{22}\le \int_{B(x,|x|/3)^c} q(t,x,y)\,dy\le \frac{c_4t}{(1+|x|)^{\alpha}}.$$
Combining all the estimates above, we obtain the desired assertion. \end{proof}

\begin{corollary}\label{c2-2}
There exist positive constants $C_5$, $C_6$ and $C_7$ such that for all $x\in \R^d$ and $t\ge C_5t_0(|x|)$,
\begin{equation}\label{l4-2-3a-}
T_{t}^V 1(x)\le C_6
\log(2+|x|)(1+|x|)^{-\alpha}\le C_7(1+|x|)^{-\alpha/2}.
\end{equation}
\end{corollary}
\begin{proof}
According to \eqref{est1}, for every $\tilde C_0>0$ and $x\in \R^d$,
\begin{equation}\label{c2-2-1}
T_{\tilde C_0t_0(|x|)}^V1(x)\le c_1\left(\exp\left(-c_2\tilde C_0t_0(|x|)g(|x|)\right)+\frac{\tilde C_0t_0(|x|)}{(1+|x|)^\alpha}\right).
\end{equation}

By \eqref{e1-6} and \eqref{l4-2-1b}, it holds that for all $x\in \R^d$,
\begin{equation}\label{c2-2-2}
\begin{split}
 \exp\left(-c_{2}\tilde C_{0}t_0(|x|)g(|x|)\right)&=
\left(e^{-t_0(|x|)g(|x|)}\right)^{c_{2}\tilde C_{0}}\\
&=\left(\frac{t_0(|x|)}{(1+|x|)^\alpha}\right)^{c_{2}\tilde C_{0}}
\le c_{3}[\log(2+|x|)]^{c_2\tilde C_0}(1+|x|)^{-c_{2}\tilde C_{0}\alpha}.
\end{split}
\end{equation}
On the other hand, by \eqref{l4-2-1b} again,
\begin{align*}
\frac{t_0(|x|)}{(1+|x|)^\alpha}\le c_4\log(2+|x|)(1+|x|)^{-\alpha}.
\end{align*}
Choosing $\tilde C_0>c_2^{-1}$ and putting both estimates above into
\eqref{c2-2-1} yield that
$$
T_{\tilde C_0t_0(|x|)}^V1(x)\le c_5\log(2+|x|)(1+|x|)^{-\alpha}.
$$

Therefore, for every $t\ge \tilde C_0t_0(|x|)$,
\begin{align*}
T_t^V 1(x)&=T_{\tilde C_0t_0(|x|)}^V\left(T_{t-\tilde C_0t_0(|x|)}^V 1\right)(x)\\
&\le T_{\tilde C_0t_0(|x|)}^V 1(x)\le c_5\log(2+|x|)(1+|x|)^{-\alpha}.
\end{align*} The proof of the first inequality in the desired assertion is complete. The second assertion is a direct consequence of the first one. \end{proof}

\begin{remark}
According to \eqref{e1-6}, for every $c_0>0$, $x\in \R^d$ and $0<t\le c_0t_0(|x|)$,
\begin{equation}\label{r2-1-1}
\frac{t}{(1+|x|)^{\alpha}}=c_0\frac{c_0^{-1}t}{(1+|x|)^{\alpha}}\le
c_0\exp\left(-c_0^{-1}tg(|x|)\right).
\end{equation}
Thus, by \eqref{est1}, for any $c_0>0$, there exist positive constants  $c_1$ and $c_2$
(which may depend on $c_0$) such that for any $x\in \R^d$ and $0<t\le c_0t_0(|x|)$,
\begin{equation}\label{e:note2}
T_t^V 1(x)\le  c_1 \exp(-c_2tg(|x|)).
\end{equation}
\end{remark}

At the end of this section, we prove an upper bound for $p(t,x,y)$ when $|x-y|\ge Ct^{1/\alpha}$.
For this we need the following representation of $p(t,x,y)$. Since the proof is almost the same as that \cite[Lemma 2.1]{CW1}
(it only uses the strong Markov property of the process $X$ and the Feynman--Kac formula \eqref{e1-4}), we omit it here.

\begin{lemma}\label{l2-6}
Suppose $U$ is a domain in $\R^d$. Then, for every $x\in U$ and $y\notin \bar U$,
\begin{equation}\label{l2-4-2}
p(t,x,y)=\Ee_x\left[\exp\left(-\int_0^{\tau_U}V(X_s)\,ds\right)\I_{\{\tau_U\le t\}}p\left(t-\tau_U,X_{\tau_U},y\right)\right].
\end{equation}
\end{lemma}

\begin{lemma} \label{L:2.1}
For any $C_0'>0$ there is a constant $C_{8}>0$ such that for all $x,y\in \R^d$ and $t>0$ with $|x-y|\ge 2C_0't^{1/\alpha}$,
\begin{equation}\label{l2-5-1}
p(t,x,y)\le  \frac{C_{8}t}{|x-y|^{d+\alpha}}\left(1\wedge\frac{1}{ t\max\{g(|x|),g(|y|)\}}\right).
\end{equation}
\end{lemma}

\begin{proof}
Without loss of generality throughout the proof we assume that
$|x|\le |y|$.
Fix any $C_0'>0$, and let $x,y\in \R^d$ and $t>0$ be such that $|x-y|\ge 2 C_0' t^{1/\alpha}$. Define $U=B(y, c_0t^{1/\alpha})$
with $c_0 := C_0'/6$ (in particular, $c_0t^{1/\alpha}= (2C_0't^{1/\alpha})/12 \leq |x-y|/12$).

We first consider the case that $|y|\ge 3$. Set $W=B(y,|x-y|/3)$. Since $|x-y|\ge 12c_0t^{1/\alpha}$, it holds that $U\subset W$. Hence, by \eqref{l2-4-2},
\begin{align*}p(t,x,y)&=\Ee_y\left[\exp\left(-\int_0^{\tau_U} V(X_s)\,ds\right)p(t-\tau_U, X_{\tau_U},x)\I_{\{\tau_U\le t\}}\right]\\
&=\Ee_y\left[\exp\left(-\int_0^{\tau_U} V(X_s)\,ds\right)p(t-\tau_U, X_{\tau_U},x)\I_{\{\tau_U\le t, X_{\tau_U}\in W\setminus U\}}\right]\\
&\quad+ \Ee_y\left[\exp\left(-\int_0^{\tau_U} V(X_s)\,ds\right)p(t-\tau_U, X_{\tau_U},x)\I_{\{\tau_U\le t, X_{\tau_U}\notin W\}}\right]\\
&=:I_1+I_2.\end{align*}
By the inclusions $U\subset W\subset B(y,{2|y|}/{3})$ (the second one follows from ${|x-y|}/{3}\le ({|x|+|y|})/{3}\le {2|y|}/{3}$), \eqref{e1-2} and \eqref{e1-3}, we have
\begin{align} \label{e:pot_to_g}
\inf_{v\in U}V(v)\ge \inf_{v\in W}V(v)\ge \inf_{v\in B(y,{2|y|}/{3})}V(v)\ge c_{1}g(|y|),\quad |y|\ge 3.
\end{align}
Therefore, by the L\'evy system formula \eqref{l2-1-1}, \eqref{e:pot_to_g} and the estimates
\begin{align*}
q_{U}(s,y,z)\le
c_2\left(s^{-d/\alpha}\wedge\frac{s}{|y-z|^{d+\alpha}}\right)\left(\frac{\delta_U(z)^{\alpha/2}}{\sqrt{s}}\wedge1\right),\quad  0<s\le t,\ z\in U,
\end{align*}
\begin{align*}
p(t-s,u,x)\le q(t-s,u,x)\le \frac{c_3t}{|u-x|^{d+\alpha}}\le \frac{c_4t}{|x-y|^{d+\alpha}},\quad 0<s\le t, \  u\in W,
\end{align*}
(cf.\ \eqref{e:heat2} and \eqref{e:heat1}), we get \normal
\begin{align*}
I_1\le & c_5\int_0^t \int_{W\setminus U}\left(\int_U q_U(s,y,z)p(t-s,u,x) \frac{1}{|z-u|^{d+\alpha}} e^{- s\inf_{v\in U}V(v)}\,dz\right)\,du\,ds\\
\le& \frac{c_6t}{|x-y|^{d+\alpha}}\int_0^t\int_{W\setminus U}\int_U\left(s^{-d/\alpha}\wedge \frac{s}{|y-z|^{d+\alpha}}\right)\left(\frac{\delta_{U}(z)^{\alpha/2}}{\sqrt{s}}\wedge1\right)\frac{e^{- c_1sg(|y|)}}{|z-u|^{d+\alpha}} \,dz\,du\,ds\\
\le & \frac{c_6t}{|x-y|^{d+\alpha}}\int_0^t \int_U\left(s^{-d/\alpha}\wedge \frac{s}{|y-z|^{d+\alpha}}\right)\left(\frac{\delta_U(z)^{\alpha/2}}{\sqrt{s}}\wedge1\right) e^{- c_1sg(|y|)} \\
& \times\int_{B(z,\delta_U(z))^c}\frac{1}{|z-u|^{d+\alpha}}\,du\,dz\,ds \\
\le & \frac{c_7t}{|x-y|^{d+\alpha}}\int_0^t\int_U \left(s^{-d/\alpha}\wedge \frac{s}{|y-z|^{d+\alpha}}\right)\delta_U^{-\alpha}(z)\left(\frac{\delta_U(z)^{\alpha/2}}{\sqrt{s}}\wedge1\right) e^{- c_1sg(|y|)}\,dz\,ds.
\end{align*}
Take now $c_8:=c_0/4$. In particular, $\delta_U(z)\ge c_{9}t^{1/\alpha}$ for all $z\in B(y,c_8 s^{1/\alpha})$, and the expression on the right hand side above is less or equal to the sum
\begin{align*}
&\frac{c_{10}t}{|x-y|^{d+\alpha}}\int_0^t\int_{B(y,c_{8} s^{1/\alpha})} s^{-d/\alpha}t^{-1} e^{- c_1sg(|y|)}\,dz\,ds\\
& +\frac{c_{10}t}{|x-y|^{d+\alpha}}\int_0^t\int_{B(y,c_0t^{1/\alpha}-c_{8} s^{1/\alpha})\backslash B(y,c_{8}s^{1/\alpha})} \frac{s}{|y-z|^{d+\alpha}}\delta_U^{-\alpha}(z) e^{-c_1sg(|y|)}\,dz\,ds \\
& +\frac{c_{10}t}{|x-y|^{d+\alpha}}\int_0^t\int_{B(y,c_0t^{1/\alpha}) \backslash B(y,c_0t^{1/\alpha}-c_{8} s^{1/\alpha})} \frac{s}{|y-z|^{d+\alpha}}\frac{\delta_U^{-\alpha/2}(z)}{\sqrt{s}} e^{-c_1sg(|y|)}\,dz\,ds\\
&=:I_{11}+I_{12}+I_{13}.
\end{align*} It is clear that
$$I_{11}\le \frac{c_{11}}{|x-y|^{d+\alpha}} \int_0^t e^{-c_1sg(|y|)}\,ds\le \frac{c_{12}t}{|x-y|^{d+\alpha}} \left(\frac{1}{tg(|y|)}\wedge 1\right).$$
On the other hand, for every $0\le s\le t$,
\begin{align*}
\quad &\int_{B(y,c_0t^{1/\alpha}-c_8 s^{1/\alpha})\backslash B(y,c_8s^{1/\alpha})} \frac{\delta_U^{-\alpha}(z)}{|y-z|^{d+\alpha}}\,dz\\
&\le c_{13}\int_{c_8s^{1/\alpha}}^{c_0t^{1/\alpha}-c_8s^{1/\alpha}}  r^{-1-\alpha} (c_0t^{1/\alpha}-r)^{-\alpha}\,dr\\
&\le c_{14} t^{-1}\int_{c_8 s^{1/\alpha}}^{{c_0t^{1/\alpha}}/{3}} r^{-1-\alpha}\,dr + c_{15} t^{-(1+\alpha)/\alpha} \int_{{c_0t^{1/\alpha}}/{3}}^{c_0t^{1/\alpha}-c_8 s^{1/\alpha}} (c_0t^{1/\alpha}-r)^{-\alpha}\,dr\\
&\le \frac{c_{15}}{ts}+ \frac{c_{16}}{t^{(1+\alpha)/\alpha}}\left[t^{(1-\alpha)/\alpha}\I_{\{\alpha\in(0,1)\}}+\log \Big(1+ \frac{t}{s}\Big) \I_{\{\alpha=1\}}+ s^{(1-\alpha)/\alpha}\I_{\{\alpha\in (1,2)\}}\right]\\
&\le \frac{c_{17}}{t}\left[\frac{1}{s}+\frac{1}{t}\I_{\{\alpha\in (0,1)\}}+\frac{1}{t}\log \Big(1+ \frac{t}{s}\Big) \I_{\{\alpha=1\}}+\left(\frac{s}{t}\right)^{1/\alpha} \frac{1}{s}\I_{\{\alpha\in (1,2)\}}\right]\le \frac{c_{18}}{ts},
\end{align*}
where in the first inequality above we have used the spherical coordinates in $B(y,c_0t^{1/\alpha})$.
This yields $$I_{1 2}\le \frac{c_{19} }{|x-y|^{d+\alpha}}\int_0^t e^{- c_1sg(|y|)}\,ds\le \frac{c_{20}t}{|x-y|^{d+\alpha}}\left(1\wedge\frac{1}{tg(|y|)}\right).$$
Furthermore, we have for every $0\le s\le t$,
\begin{align*}&\int_{B(y,c_0t^{1/\alpha})\setminus B(y,c_0t^{1/\alpha}-c_8s^{1/\alpha})} \frac{\delta_U^{-\alpha/2}(z)}{|y-z|^{d+\alpha}}\,dz\\
&\le c_{21} t^{-1-d/\alpha}\int_{B(y,c_0t^{1/\alpha})\setminus B(y,c_0t^{1/\alpha}-c_8s^{1/\alpha})}\delta_U^{-\alpha/2}(z)\,dz\\
&\le c_{22} t^{-1-d/\alpha}\int_{c_0t^{1/\alpha}-c_8 s^{1/\alpha}}^{c_0t^{1/\alpha}} (c_0t^{1/\alpha}-r)^{-\alpha/2} r^{d-1}\,dr\\
&\le c_{23} t^{-1-1/\alpha} s^{1/\alpha-1/2}\le c_{23} t^{-1} s^{-1/2}.\end{align*}
This gives us that
$$I_{13}\le \frac{c_{24} }{|x-y|^{d+\alpha}} \int_0^t e^{-c_1sg(|y|)}\,ds \le \frac{c_{25}t}{|x-y|^{d+\alpha}}\left(1\wedge\frac{1}{ tg(|y|)}\right).$$

Combining with all estimates above, we arrive at
$$I_1\le \frac{c_{26}t}{|x-y|^{d+\alpha}}\left(1\wedge\frac{1}{ tg(|y|)}\right).$$

According to the L\'evy system formula  \eqref{l2-1-1} and \eqref{e:pot_to_g}, we get
\begin{align*}I_2&\le \int_0^t\int_{W^c}\int_U q_U(s,y,z)p(t-s,u,x) \frac{1}{|z-u|^{d+\alpha}}e^{- s\inf_U V}\,dz\,du\,ds\\
&\le \frac{c_{27}}{|x-y|^{d+\alpha}}\int_0^t \left(\int_U q_U(s,y,z)\,dz\right) \left( \int_{\R^d} p(t-s,u,x)\,du\right) e^{- c_{28}sg(|y|)}\,ds\\&\le \frac{c_{29}t}{|x-y|^{d+\alpha}}\left(1\wedge\frac{1}{ tg(|y|)}\right),\end{align*}
where in the second inequality we used the fact that for $z\in U$ and $u\in W^c$,
$$|z-u|\ge |u-y|-|y-z|\ge |x-y|/3-c_0t^{1/\alpha} \ge |x-y|/3-|x-y|/12 =  |x-y|/4.$$
Therefore, we have proved \eqref{l2-5-1} for the case $|y|\ge 3$.

When $|y|\le 3$, the condition $|x-y|\ge 2C_0't^{1/\alpha}$ implies that $$t\le \left(\frac{|x-y|}{2C_0'}\right)^\alpha \le c_{30}(|x|+|y|)^\alpha
\le 2^\alpha c_{30}|y|^\alpha\le c_{31},$$
and so
\begin{align*}
p(t,x,y)\le q(t,x,y)\le \frac{c_{32}t}{|x-y|^{d+\alpha}}\le \frac{c_{33}t}{|x-y|^{d+\alpha}}\left(1\wedge \frac{1}{tg(|y|)}\right).
\end{align*}
The proof is finished.
 \end{proof}

\section{Two-sided estimates when $0<t\le C_0't_0(|x|\wedge |y|)$ for any $C_0'>0$}\label{section3}
Throughout this section, without loss of generality, we assume that $x,y\in \R^d$ are such that $ |x| \le |y|$.   We begin with the following statement for $x,y\in \R^d$ with $|x-y|\le C_1't^{1/\alpha}$.

\begin{lemma}\label{P:2.1}
For any $C_0',C_1'>0$ there exist positive constants $ C_1,\ldots,C_4$
such that for all $x,y\in \R^d$ and $0<t\le C_0't_0(|x|)$ with
$|x-y|\le C_1't^{1/\alpha}$,
\begin{equation}\label{l3-1-1}
  C_{1} t^{-d/\alpha} e^{-C_{2} tg(|x|)}\le p (t,x,y)\le C_{3} t^{-d/\alpha} e^{-C_{4} tg(|x|)}.
 \end{equation}
\end{lemma}
\begin{proof}
Suppose that $x,y\in \R^d$ and $0<t\le C_0't_0(|x|)$ are such that $|x-y|\le C_1't^{1/\alpha}$. Then,
\begin{equation}\label{e:ess}
\begin{split}
p(t,x,y)&=\int_{\R^d} p(t/2,x,z)p(t/2,z,y)\,dz\le c_1t^{-d/\alpha} \int_{\R^d} p(t/2,x,z)\,dz\\
&= c_1t^{-d/\alpha} T_{t/2}^V1(x)
\le c_2 t^{-d/\alpha} e^{-c_3tg(|x|)},
\end{split}\end{equation}
where in the last inequality we used \eqref{e:note2}.

On the other hand, since $|x-y|\le C_1't^{1/\alpha}$, $y\in B(x, C_1't^{1/\alpha})$. Take
$U:= B(x,3C_1' t^{1/\alpha})$ and $f\in C_b(\R^d)$ with ${\rm supp}[f]\subset B(x,2C_1't^{1/\alpha})$. We have
\begin{align*}
T_t^Vf(x)=&\Ee_x\left[f(X_t)\exp\left(-\int_0^t V(X_s)\,ds\right)\right]\ge  \Ee_x\left[f(X_t)\exp\left(-\int_0^t V(X_s)\,ds\right)\I_{\{\tau_U>t\}}\right]\\
\ge& e^{- t\sup_{z\in U}V(z) }\Ee_x \left[f(X_t)\I_{\{\tau_U>t\}}\right]\ge e^{-c_4tg(|x|) }\int_{B(x,2C_1' t^{1/\alpha})}q_U(t,x,z) f(z)\,dz\\
\ge &c_5 t^{-d/\alpha} e^{- c_4tg(|x|)}\int_{B(x,2C_1' t^{1/\alpha})} f(z)\,dz= c_5t^{-d/\alpha} e^{-c_4tg(|x|)}\|f\|_1.\end{align*}
Here in the third inequality we used
the fact that
\begin{equation}\label{l3-1-2}
\sup_{z\in B(x,3C_0't^{1/\alpha})}V(z)\le \sup_{z\in B(x,c_6(1+|x|))}V(z)\le c_7g(|x|),
\end{equation}
which is due to that $0<t\le t_0(|x|)\le c_8(1+|x|)^\alpha$
(that can be verified by \eqref{l4-2-1b}),
and the fourth inequality follows from
\eqref{e:heat2}.
Since $f$ is arbitrary, the inequality above yields that
$$p (t,x,y)\ge c_5 t^{-d/\alpha} e^{-c_4tg(|x|)}.$$
Therefore,  the proof of \eqref{l3-1-1} is completed.
\end{proof}

\begin{lemma}\label{L:2.2} For any $C_0',C_1'>0$ there exist positive constants $C_{5}$ and $C_{6}$ such that for $x,y\in \R^d$ and $0<t\le C_0't_0(|x|)$ with $|x-y|> C_1't^{1/\alpha}$,
\begin{equation}\label{l3-2-1}
p(t,x,y)\ge C_{5}\left(\frac{1}{tg(|y|)}\wedge1\right) e^{-C_{6}tg(|x|)}\frac{t}{|x-y|^{d+\alpha}}.
\end{equation}
\end{lemma}

\begin{proof}
Recall that $|x| \le |y|$. Let $U=B(y,c_0t^{1/\alpha})$ and $W=B(x, c_0 t^{1/\alpha})$ with $c_0:=C_1'/4$. Since $C_1't^{1/\alpha}  < |x-y| \leq 2|y|$, we have $c_0 t^{1/\alpha} < |y|/2$. Consequently,
\begin{align} \label{e:pot_to_g_2}
\sup_{z\in U}V(z)\le \sup_{z\in B(y,|y|/2)}V(z)\le c_{1}g(|y|).
\end{align}
According to \eqref{l2-4-2}, the L\'evy system formula \eqref{l2-1-1}, \eqref{l3-1-1}, \eqref{e:pot_to_g_2} and the fact
that
\begin{align*}
|z-u|\le |x-y|+|z-x|+|u-y|\le |x-y|+2c_0t^{1/\alpha}\le c_2|x-y|,\quad z\in W,\ u\in U,
\end{align*}
we have
\begin{align*}p(t,y,x)=&\Ee_y \left[p(t-\tau_U, X_{\tau_U},x)\I_{\{\tau_U\le t\}} \exp\left(-\int_0^{\tau_U}V(X_s)\,ds\right)\right]\\
\ge&\Ee_y\left[p(t-\tau_U,X_{\tau_U},x) e^{-\int_0^{\tau_U} V(X_s)\,ds} \I_{\{\tau_U\le t/2,\, X_{\tau_U}\in W\}}\right]\\
\ge &c_3\left(\inf_{t/2\le s\le t \atop z\in W} p(s,z,x)\right)\int_0^{t/2}\int_U \int_{W}q_U(s,y,z) \frac{1}{|z-u|^{d+\alpha}} e^{-c_1sg(|y|)}\,dz\,du\,ds\\
\ge & c_4e^{-c_5tg(|x|)} t^{-d/\alpha} \frac{|W|}{|x-y|^{d+\alpha}}\int_0^{t/2}\int_U q_U(s,y,z)\,dze^{-c_1sg(|y|)}\,ds\\
\ge & c_6e^{-c_5tg(|x|)}\frac{1}{|x-y|^{d+\alpha}}\int_0^{t/2} \Pp_y(\tau_{B(y,c_0t^{1/\alpha})}>t/2)e^{-c_1sg(|y|)}\,ds\\
\ge & c_7 e^{-c_5t g(|x|)}\frac{1}{|x-y|^{d+\alpha}}\left(\frac{1}{g(|y|)}\wedge t\right) \\
  = &  c_7\left(\frac{1}{tg(|y|)}\wedge1\right) e^{-c_5tg(|x|)}\frac{t}{|x-y|^{d+\alpha}}.
\end{align*}
 Here we also used the property
 \begin{align*}
\Pp_y(\tau_{B(y,c_0t^{1/\alpha})}>t) = \Pp_0(\tau_{B(0,c_0t^{1/\alpha})}>t) \ge c_8>0,
 \end{align*}
 which can be verified directly by \eqref{e:heat2}. Hence, the conclusion \eqref{l3-2-1} follows.
\end{proof}
\begin{lemma}\label{L:2.3}
For any $C_0', C_1'>0$ there exist positive constants $C_{7}, C_{8}$ such that for $ x,y\in \R^d$ and $0<t\le C_0't_0(|x|)$ with $|x-y|> C_1't^{1/\alpha}$,
\begin{equation}\label{l3-3-1}
p (t,x,y)\le C_{7}\left(\frac{1}{tg(|y|)}\wedge1\right) e^{-C_{8}tg(|x|)}\frac{t}{|x-y|^{d+\alpha}}.
\end{equation}
\end{lemma}
\begin{proof}  Recall we assume that $|x| \le |y|$. The proof is split into two cases.

\smallskip

\noindent
{\bf Case 1:\ $|y-x|\le |y|/4$.} In particular, $|x|\ge |y|-|y-x|\ge |y|/2$.
We write
\begin{align*}
p(t,x, y)&=\int_{\{|z-y|\le |x-y|/2\}} p(t/2,x,z)p(t/2,z,y)\,dz +\int_{\{|z-y|\ge |x-y|/2\}} p(t/2,x,z)p(t/2,z,y)\,dz\\
&=:I_1+I_2.
\end{align*}

When $|y-z|\le |x-y|/2$, we have $|z-x|\ge |x-y|-|y-z|\ge |x-y|/2\ge {C_1't^{1/\alpha}}/{2}$ and $|z|\ge |y|-|y-z|\ge |y|-|x-y|/2\ge |y|-|y|/8\ge 7|y|/8$.
Then, according to \eqref{l2-5-1} and   \eqref{e1-2},
\begin{align*}
p(t/2,x,z)&\le c_1\left(\frac{1}{t\max\{g(|x|),g(|z|)\}}\wedge1\right) \frac{t}{|x-z|^{d+\alpha}}\le c_2\left(\frac{1}{tg(|y|)}\wedge1\right) \frac{t}{|x-y|^{d+\alpha}}.
\end{align*}
So for any $0<t\le C_0't_0(|x|)$ it holds that
\begin{align*}I_1\le &c_2\left(\frac{1}{tg(|y|)}\wedge1\right) \frac{t}{|x-y|^{d+\alpha}}\int_{\R^d} p(t/2,z,y)\,dz\\
 \le &c_3 \left(\frac{1}{tg(|y|)}\wedge1\right) \frac{t}{|x-y|^{d+\alpha}}\left(e^{-c_4t g(|y|)}+\frac{t}{(1+|y|)^{\alpha}} \right)\\
\le& c_5 \left(\frac{1}{tg(|y|)}\wedge1\right) \frac{t}{|x-y|^{d+\alpha}}\left(e^{-c_6t g(|x|)}+\frac{t}{(1+|x|)^{\alpha}} \right)\\
\le & c_7 \left(\frac{1}{tg(|y|)}\wedge1\right) \frac{t}{|x-y|^{d+\alpha}} e^{-c_8tg(|x|)},
\end{align*}
where the second inequality follows from \eqref{est1}, in the third inequality we have used the fact $|x|\le |y|$, and the last inequality is due to
  the fact $0<t\le C_0't_0(|x|)$
  and \eqref{r2-1-1}.

Similarly, by \eqref{l2-5-1}, for every $z,y\in \R^d$ with $|z-y|\ge |x-y|/2\ge  {C_1't^{1/\alpha}}/{2}$,
\begin{align*}
p(t/2, z,y)&\le c_9\left(\frac{1}{t\max\{g(|y|),g(|z|)\}}\wedge1\right)\frac{t}{|z-y|^{d+\alpha}}\le c_{10}\left(\frac{1}{tg(|y|)}\wedge1\right) \frac{t}{|x-y|^{d+\alpha}}.
\end{align*}
Hence, due to
\eqref{e:note2}, it holds that for every $0<t\le C_0t_0(|x|)$,
\begin{align*}
I_2\le & c_{10}\left(\frac{1}{tg(|y|)}\wedge1\right) \frac{t}{|x-y|^{d+\alpha}}\int_{\R^d} p(t/2,x,z)\,dz \le c_{11}\left(\frac{1}{tg(|y|)}\wedge1\right) \frac{t}{|x-y|^{d+\alpha}} e^{-c_{12}tg(|x|)}.
\end{align*}
Therefore, putting all the estimates above for $I_1$ and $I_2$ together yields
$$p (t,x,y)\le   c_{13}\left(\frac{1}{tg(|y|)}\wedge1\right) \frac{t}{|x-y|^{d+\alpha}} e^{-c_{14} tg(|x|)}.$$

\smallskip

\noindent
{\bf Case 2:\ $|y-x|> |y|/4$.}
 Since we assume that $|x|\le |y|$, it holds that $|y|/4\le |y-x|\le |y|+|x|\le 2|y|$. Then, we set
\begin{align*}
p(t,x, y)&=\int_{\{|z-y|\le |y|/16\}} p(t/2,x,z)p(t/2,z,y)\,dz +\int_{\{|z-y|\ge |y|/16\}} p(t/2,x,z)p(t/2,z,y)\,dz\\
&=:J_1+J_2.
\end{align*}
When $|z-y|\le |y|/16$, we have $|z-x|\ge |x-y|-|y-z|\ge |x-y|-|y|/16\ge |x-y|-|x-y|/4\ge {3|x-y|}/{4}\ge {3C'_1t^{1/\alpha}}/{4}$ and
$|z|\ge |y|-|z-y|\ge |y|/2$. Then, according to \eqref{l2-5-1} and \eqref{e1-2},
$$p(t/2,z,x)\le  c_{15}\left(\frac{1}{t\max\{g(|x|),g(|z|)\}}\wedge 1\right)\frac{t}{|x-z|^{d+\alpha}}\le c_{16}\left( \frac{1}{tg(|y|)}\wedge1\right)\frac{t}{|x-y|^{d+\alpha}}.$$
Thus, following the argument for $I_1$ and using \eqref{e:note2},
we find that
\begin{align*}
J_1\le &c_{16}\left( \frac{1}{tg(|y|)}\wedge1\right)\frac{t}{|x-y|^{d+\alpha}}\int_{\R^d} p(t/2,z,y)\,dz\le c_{17}\left( \frac{1}{tg(|y|)}\wedge1\right)\frac{t}{|x-y|^{d+\alpha}} e^{-c_{18} tg(|x|)}.
\end{align*}

Note that, when $|z-y|\ge |y|/16$,
$|z-y|\ge |x-y|/32\ge {C'_1t^{1/\alpha}}/{32}.$ Hence, according to \eqref{l2-5-1}, in this case we have
$$p(t/2,z,y)\le c_{19}\left(\frac{1}{t\max\{g(|z|),g(|y|)\}}\wedge1\right)\frac{t}{|z-y|^{d+\alpha}}\le c_{20}\left(\frac{1}{tg(|y|)}\wedge1\right)\frac{t}{|x-y|^{d+\alpha}}.$$
Then, applying \eqref{e:note2} again, we deduce that
\begin{align*}J_2\le &c_{20}\left(\frac{1}{tg(|y|)}\wedge1\right)\frac{t}{|x-y|^{d+\alpha}}\int_{\R^d} p (t/2,x,z)\,dz\le  c_{21}\left(\frac{1}{tg(|y|)}\wedge1\right)\frac{t}{|x-y|^{d+\alpha}}e^{-c_{22} tg(|x|)}.\end{align*}

Putting the estimates for $J_1$ and $J_2$ together, we arrive at
$$p(t,x,y)\le c_{23} \left(\frac{1}{tg(|y|)}\wedge1\right)\frac{t}{|x-y|^{d+\alpha}} e^{-c_{24} tg(|x|)}.$$
Therefore, the assertion follows.
\end{proof}

Now we can summarise the results in this section as follows.

\begin{proposition}\label{P:2.4}
For any $C_0'>0$ there exist positive constants  $C_9,\ldots,C_{12}$  such that
for all $x,y\in \R^d$ with $0<t\le C_0't_0(|x|)$,
\begin{equation}\label{p3-1-1}
\begin{split}
C_{9}\left(\frac{1}{tg(|y|)}\wedge1\right) e^{-C_{10}tg(|x|)} & \left(t^{-d/\alpha}\wedge \frac{t}{|x-y|^{d+\alpha}}\right) \le p(t,x,y) \\
& \le C_{11}\left(\frac{1}{tg(|y|)}\wedge1\right) e^{-C_{12}tg(|x|)}\left(t^{-d/\alpha}\wedge \frac{t}{|x-y|^{d+\alpha}}\right).
\end{split}
\end{equation}
\end{proposition}
\begin{proof}
Let $C_1'>0$. When $x,y\in \R^d$ and $0<t\le C_0't_0(|x|)$ are such that $|x-y|\ge C_1't^{1/\alpha}$, the assertion follows directly
from Lemmas \ref{L:2.2} and \ref{L:2.3}. On the other hand, by   Lemma \ref{P:2.1},  for all $x,y\in \R^d$ with $|x-y|\le C_1't^{1/\alpha}$ and $0<t\le C_0't_0(|x|)$,
$$p(t,x,y)\ge  c_1e^{- c_2tg(|x|)}t^{-d/\alpha}\ge c_1 \left(\frac{1}{tg(|y|)} \wedge1\right) e^{- c_2tg(|x|)}t^{-d/\alpha}.$$
Furthermore, for all $x,y\in \R^d$ with $|x-y|\le C_1't^{1/\alpha}$,
\begin{equation}
\begin{split}
g(|x|)&\le g(|y|) \le g(|x|+|x-y|) \le g(|x|+C_1't^{1/\alpha})\\
&\le g(|x|+C_1'C_0'^{1/\alpha}t_0(|x|)^{1/\alpha})\le
g(c_3(1+|x|))\le c_4g(|x|),
\end{split}
\end{equation}
and so
\begin{equation}\label{p3-1-2}
\frac{1}{tg(|y|)} \wedge 1\ge c_5\left(\frac{1}{tg(|x|)} \wedge 1\right)\ge c_6e^{-c_7tg(|x|)},
\end{equation}
where we used the monotonicity  of $g(\cdot)$ and the fact that $t_0(|x|)^{1/\alpha}\le c_8(1+|x|)$.
Thus, using \eqref{p3-1-2} and  Lemma \ref{P:2.1} again, we obtain that for  all $x,y\in \R^d$ with $|x-y|\le C_1't^{1/\alpha}$ and $0<t\le C_0't_0(|x|)$,
$$p(t,x,y)\le  c_{9}t^{-d/\alpha}e^{- c_{10}tg(|x|)}\le c_{11} t^{-d/\alpha}\left(\frac{1}{tg(|y|)} \wedge1\right) e^{- c_{12}tg(|x|)}.$$
Putting all estimates  above together yields the desired assertion \eqref{p3-1-1}.
\end{proof}

\section{Two-sided estimates when $t\ge C_0t_0(|x|\wedge |y|)$
for some $C_0>0$}\label{section4}

In this section, we will present two-sided estimates of $p(t,x,y)$ when $t>C_0t_0(|x|\wedge |y|)$ for some properly chosen
constant $C_0>0$. The statements are split into two parts, according to two
different almost monotonicity properties of $t_0(\cdot)$. Without loss of generality, we still assume that $|x| \le |y|$.
Recall that
$$H(x):=\frac{1}{g(|x|)(1+|x|)^{d+\alpha}},\quad x\in \R^d.$$

We first give a lemma for a rough upper bound of the heat kernel $p(t,x,y)$ when $t>C_0t_0(|x|)$ for some $C_0>0$.
\begin{lemma}\label{l4-5}
There exist $N_1, C_1>0$ such that for all $z,u\in \R^d$ and $t>0$ with $t>N_1C_5t_0(|z|)$,
\begin{equation}\label{l4-5-1}
p(t,z,u)\le C_1H(z),
\end{equation} where $C_5>0$ is the constant given in Corollary $\ref{c2-2}$.
\end{lemma}
\begin{proof}
Let $R_0>0$. For every $|z|\le R_0$ and $t>   C_5 t_0(|z|)$, we have $t\ge c_1$ by \eqref{l4-2-1b}, and so
\begin{align*}
p(t,z,u)\le c_2t^{-d/\alpha}\le \frac{c_3}{g(|z|)(1+|z|)^{d+\alpha}},\quad z,u\in \R^d\ {\rm with}\ |z|\le R_0,\ t>   C_5 t_0(|z|).
\end{align*}
 Here $C_5>0$ is the constant from Corollary $\ref{c2-2}$.
Next, we prove \eqref{l4-5-1} for every $|z|>R_0$ with $R_0$ large enough so that
$$\inf_{|z|>R_0}\left(|z|^\alpha-C_5t_0(|z|)\right)>0;$$   we note that such $R_0$ can be found due to \eqref{l4-2-1b}.

For $C_5t_0(|z|)\le t \le |z|^\alpha$,
 \begin{equation}\label{l4-2-5}
 \begin{split}
 p(2t,z,u)\le
 &\begin{cases} \displaystyle\int_{\R^d}
 p(t,z,w)p(t,w,u)\,dw,&\quad |u|\ge |z|/4\\
 c_{4}\left(\frac{1}{tg(|z|)}\wedge1\right)\frac{t}{|z-u|^{d+\alpha}},&\quad |u|<|z|/4\end{cases}\\
 \le&  \begin{cases}
 c_{5}t^{-d/\alpha} \displaystyle\int_{\R^d} p(t,z,w)\,dw,&\quad |u|\ge |z|/4\\
\frac{c_{4}}{g(|z|)(1+|z|)^{d+\alpha}},&\quad |u|<|z|/4
\end{cases}\\
 \le&  \begin{cases}  c_{6}(1+|z|)^{{dq_1}/{\alpha}} (1+|z|)^{-\alpha/2},&\quad |u|\ge |z|/4\\
 \frac{c_{4}}{g(|z|)(1+|z|)^{d+\alpha}},&\quad |u|<|z|/4\end{cases}\\
 \le& \frac{c_{7}}{g(|z|)(1+|z|)^{d+\alpha}}\left(1+(1+|z|)^{ {dq_1}/{\alpha}+d+\alpha+q_1-\alpha/2}\right).
 \end{split}
 \end{equation}
  Here in the first inequality we have used Lemma \ref{L:2.1} and the fact that  $|z-u|\ge |z|/2\ge t^{1/\alpha}/2$
  for every $|u|<|z|/4$
  and $t\le |z|^\alpha$, the third inequality
   follows from \eqref{l4-2-3a-}, and the fact $t>C_5t_0(|z|)\ge c_8(1+|z|)^{-q_1}$ (which
   has been established by \eqref{l4-2-1b}), and the last inequality is due to \eqref{l4-2-1}.

Furthermore, according to \eqref{l4-2-5}, we get
\begin{align*}
p(3t,z,u)\le&
\begin{cases}\displaystyle\int_{\R^d} p(2t,z,w)p(t,w,u)\,dw,&\quad |u|\ge |z|/4\\
c_{9}\left(\frac{1}{tg(|z|)}\wedge1\right)\frac{t}{|z-u|^{d+\alpha}},&\quad |u|<|z|/4\end{cases}\\
 \le&  c_{10}\begin{cases}
 \frac{(1+(1+|z|)^{{dq_1}/{\alpha}+d+\alpha+q_1-\alpha/2})}{g(|z|)(1+|z|)^{d+\alpha}}\displaystyle\int_{\R^d}p(t,w,u)\,dw ,&\quad |u|\ge |z|/4\\
 \frac{1}{g(|z|)(1+|z|)^{d+\alpha}},&\quad |u|<|z|/4\end{cases}\\
 \le&  c_{11}\begin{cases}  \frac{1}{g(|z|)(1+|z|)^{d+\alpha}}(1+(1+|z|)^{{dq_1}/{\alpha}+d+\alpha+q_1-\alpha}) ,&\quad |u|\ge |z|/4\\
 \frac{1}{g(|z|)(1+|z|)^{d+\alpha}},&\quad |u|<|z|/4\end{cases}\\
 \le& \frac{c_{12}}{g(|z|)(1+|z|)^{d+\alpha}}\left(1+(1+|z|)^{{dq_1}/{\alpha}+d+\alpha+q_1-\alpha}\right),\end{align*}
where in the first inequality we used Lemma \ref{L:2.1} (similarly as in the proof of the first inequality in \eqref{l4-2-5}), the second inequality follows from
 \eqref{l4-2-5}, and the third inequality is due to the fact that for all $|u|\ge |z|/4$ and $C_5t_0(|z|)\le t \le |z|^\alpha$,
\begin{align*}\int_{\R^d}p(t,w,u)\,dw =&T_t^V1(u)=T_{C_5t_0(|z|)}\left(T_{t-C_5t_0(|z|)}1\right)(u)\le T_{C_5t_0(|z|)} 1(u)\\
\le & c_{13}\left[\exp(-c_{14}C_5t_0(|z|)g(|u|))+\frac{C_5t_0(|z|)}{(1+|u|)^\alpha}\right]\\
\le & c_{15}\left[\exp(-c_{16}C_5t_0(|z|)g(|z|))+\frac{C_5t_0(|z|)}{(1+|z|)^\alpha}\right]\le c_{17}(1+|z|)^{-\alpha/2}\end{align*} thanks to \eqref{est1} and the arguments used for
 \eqref{c2-2-2}
(note that one can take the constant $C_5$ large enough to ensure the procedure in \eqref{c2-2-2}).

Now, by repeating the procedure above for $N_1$ times, where ${dq_1}/{\alpha}+d+\alpha+q_1-(N_1-1)\alpha/2<0$, we obtain that
 \begin{equation}\label{e:lf1}
 p(N_1t,z,u)\le \frac{c_{18}}{g(|z|)(1+|z|)^{d+\alpha}},\quad C_5t_0(|z|)\le t \le |z|^\alpha.
 \end{equation}

Hence, for any
$N_1|z|^\alpha\ge t\ge N_1C_5t_0(|z|)$, we can find $t_1\in [ C_5t_0(|z|), |z|^\alpha]$ such that $t=N_1t_1$. Then, by \eqref{e:lf1},
$$p(t,z,u)= p(N_1t_1,z,u)\le \frac{c_{18}}{g(|z|)(1+|z|)^{d+\alpha}}.$$ On the other hand, for any $t>N_1|z|^\alpha$,
\begin{align*}
p(t,z,u)&=\int_{\R^d}p(N_1|z|^\alpha,z,w)p\left(t-N_1|z|^\alpha,w,u\right)\,dw\\
&\le \frac{c_{19}}{g(|z|)(1+|z|)^{d+\alpha}}\cdot \int_{\R^d}p\left(t-N_1|z|^\alpha,w,u\right)\,dw\\
&\le \frac{c_{20}}{g(|z|)(1+|z|)^{d+\alpha}}=c_{20}H(z).
\end{align*}
The proof is complete.
\end{proof}

\subsection{$t_0(\cdot)$ is almost  decreasing}\label{subsection4.1}
In this subsection, we always assume that $t_0(\cdot)$ is almost decreasing,   i.e.,\ there is a
continuous and strictly
decreasing function $h: [0,\infty)\to (0,\infty)$ such that  \eqref{e4-1} is satisfied.
Define
$$s_0(t)=h^{-1}(t)\vee 2,\quad t\ge0,$$ where
$
h^{-1}(t):=\inf\{ s \ge 0:   h(s)\le t\}
$ and here we use the convention that $\inf \emptyset=-\infty$.
Below, set   $
T_0:=h(2) > 0$ and $T_{\infty}:=h(\infty)=\lim_{s \to \infty} h(s) \geq 0.$
As $h(\cdot)$ is
continuous and strictly
decreasing, it is easy to verify that
\begin{equation}\label{e4-3}
s_0(t)=h^{-1}(t)\ge 2\quad \hbox{for}\,\, t\in  (T_{\infty},T_0];
\end{equation}
moreover, for every $t\in (T_0, \infty)$,
and
for every $t \in (0,T_{\infty}]$
when $T_{\infty}>0$,
it holds that $s_0(t)=2$.

\begin{lemma}\label{L:3.1}
For every sufficiently large $C_0>0$ there are constants $C_2,C_3>0$ such that for all $x,y\in \R^d$ and $C_0 h(|x|)< t \le C_0T_0$, \normal
it holds that
\begin{equation}\label{l4-1-1}
p(t,x,y)\ge C_2 H(x)H(y)\int_{\{|z|\le s_0(t)\}}e^{-C_3tg(|z|)}\,dz.
\end{equation}
\end{lemma}
\begin{proof}
Let $C_0 \geq C_{**}$, where $C_{**}$ is the constant in \eqref{e4-1}, and assume first that $C_0 h(|x|)< t \le T_0$.
We have
$$
T_0=h(2)\ge t> C_0 h(|x|)\ge C_0 C_{**}^{-1} h(|x|/2)\ge h(|x|/2) > h(\infty)=T_{\infty}.
$$
This, along with \eqref{e4-3} and the fact that $h(\cdot)$ is
continuous and strictly
decreasing, yields that $|x|\ge 2h^{-1}(t)=2 s_0(t)\ge4$
for every $C_0 h(|x|)< t \le T_0$.
On the other hand, for every $T_{\infty} < t\le T_0$ and $z\in \R^d$ with $|z|\le s_0(t)$, by \eqref{e4-1} and \eqref{e4-3}, it holds that
\begin{equation}\label{e:note1}
t_0(|z|)\ge c_1h(|z|)\ge   c_1 h(s_0(t))=c_1t.
\end{equation}
Hence, according to \eqref{p3-1-1} and \eqref{e:note1}, for every $z\in \R^d$ with $|z|\le s_0(t)$ (also by noting here that
 $|z|\le |x|/2\le |y|/2$ since $|x|\ge 2s_0(t)$)
\begin{align*}
p(t/2,z,x)\ge & c_2\left(\frac{1}{tg(|x|)}\wedge 1\right) e^{-c_3 tg(|z|)}\frac{t}{|z-x|^{d+\alpha}}\ge \frac{c_4}{g(|x|)}e^{-c_3 tg(|z|)}\frac{1}{(1+|x|)^{d+\alpha}},
\end{align*}
where we used the facts that for all $C_0 h(|x|)<t\le T_0$,
\begin{align*}
c_5t^{1/\alpha}\le 2\le s_0(t)\le |x|/2\le|z-x|\le c_6(1+|x|),\quad |x|\ge 2 s_0(t)\hbox{ and } |z|\le s_0(t)
\end{align*} and \begin{equation}\label{e:sss}
tg(|x|)\ge c_7t_0(|x|)g(|x|)\ge c_8,
\end{equation}
thanks to \eqref{e1-6}  (see
the second to last
line in the proof of Lemma \ref{dbl-polyn}).

Since $|y|\ge |x|\ge 2s_0(t)$, similarly we can obtain that for every $z\in \R^d$ with $|z|\le s_0(t)$,
\begin{align*}
p(t/2, z,y)\ge &c_9 \left(\frac{1}{tg(|y|)}\wedge1\right)e^{-c_{10} tg(|z|)}\frac{t}{|y-z|^{d+\alpha}}\ge  \frac{c_{11}}{g(|y|)}  e^{-c_{12} tg(|z|)}\frac{1}{(1+|y|)^{d+\alpha}}.
\end{align*}

Therefore,  for all $C_0 h(|x|)\le t \le T_0$,
\begin{align*}
p(t,x,y)&\ge \int_{\{|z|\le s_0(t)\}}p(t/2,x,z)p(t/2, z,y)\,dz\\
&\ge  \frac{c_{13}}{g(|x|)g(|y|)(1+|x|)^{d+\alpha}(1+|y|)^{d+\alpha}}\int_{\{|z|\le s_0(t)\}} e^{-c_{14} tg(|z|)}\,dz,
\end{align*}
which is the claimed estimate.

Suppose now that $t > C_0 h(|x|)$ and $ T_0 < t \le C_0T_0$. Since $s_0(t) = 2$ for this range of $t$, we can follow the argument above to show the same estimate when $|y| \geq |x| \geq 4 = 2s_0(t)$. When $|x| < 4$ and $|y| \geq |x|$, by the fact $t\le C_0T_0\le c_{15}t_0(|x|)$ (since $|x|<4$), we can directly apply \eqref{p3-1-1} to obtain
\begin{align*}
p(t,x,y)&\ge
c_{16}\left(\frac{1}{tg(|y|)}\wedge 1\right)e^{-c_{17}tg(|x|)}\left(t^{-d/\alpha}\wedge \frac{t}{|x-y|^{d+\alpha}}\right)\\
& \ge \frac{c_{18}}{g(|x|)g(|y|)(1+|x|)^{d+\alpha}(1+|y|)^{d+\alpha}} \\
        & \ge \frac{c_{18}}{g(|x|)g(|y|)(1+|x|)^{d+\alpha}(1+|y|)^{d+\alpha}} \int_{\{|z|\le s_0(t)\}} e^{-c_{17} tg(|z|)}\,dz.
\end{align*}
Here in the second inequality we have used the facts
$$
g(|x|)\asymp (1+|x|)\asymp 1,\quad |x|\le 4
$$ and
$$
\left(t^{-d/\alpha}\wedge \frac{t}{|x-y|^{d+\alpha}}\right)\ge c_{20}(1+|y|)^{-d-\alpha},\quad T_0<t \le C_0T_0,\ |y|\ge |x|,\ |x|\le 4,
$$ and
the last inequality is due to the property
\begin{align*}
\int_{\{|z|\le s_0(t)\}} e^{-c_{17} tg(|z|)}\,dz=\int_{\{|z|\le 2\}} e^{-c_{17} tg(|z|)}\,dz\le c_{21},\quad T_0<t \le C_0T_0.
\end{align*}
By now the proof of \eqref{l4-1-1} is complete.
\end{proof}

\begin{lemma}\label{L:3.2}
For sufficiently large $C_0>0$ there are constants $C_4, C_5,C_6$  such that for every $C_0 h(|x|)< t \le C_0T_0$,
\begin{equation}\label{l4-2-2}
p(t,x,y)\le C_4H(x)H(y)
\int_{\{|z|\le s_0(C_5t)\}}e^{-C_6tg(|z|)}\,dz.
\end{equation}
\end{lemma}
\begin{proof}
According to \eqref{e4-1},
\eqref{l4-2-3a-} and \eqref{l4-5-1},
there exists a constant
$\tilde C_0>1$
such that
\begin{equation}\label{l4-2-4}
T_{t}^V1(z)\le c_{1}(1+|z|)^{-\alpha/2},\quad z\in \R^d,\ t>0\ {\rm with}\ t\ge\tilde C_0h(|z|)
\end{equation} and
\begin{equation}\label{l4-2-3a}
p(t,z,u)\le c_{1}H(z),\quad u,z\in \R^d,\ t>0\ {\rm with}\ t\ge \tilde C_0h(|z|).
\end{equation}

 Furthermore, by \eqref{e4-1} and the fact that $h(\cdot)$ is continuous and strictly decreasing,
as well as that $s_0(t)=h^{-1}(t)$  when
$T_{\infty} \le h(|x|) < t\le T_0$,
we obtain that for all
$\tilde C_0h(|x|)< t \le \tilde C_0 T_0$
the following properties hold:
\begin{equation}\label{l4-2-3}
\begin{split}
t\le \tilde C_0h(|z|),\ &\quad  z\in \R^d \,\,{\rm   with}\ |z|\le s_0\left(\frac{t}{\tilde C_0}\right)=:\tilde s_0(t),\\
t > \tilde C_0h(|z|),\ &\quad z\in \R^d\,\, {\rm   with}\ |z|>s_0\left(\frac{t}{\tilde C_0}\right)=:\tilde s_0(t).
\end{split}
\end{equation}

Now, for every
$\tilde C_0h(|x|)< t \le \tilde C_0 T_0$
and $|z|>\tilde s_0(t)$,
\begin{equation}\label{e:lf2}
\begin{split}
\quad p(2t,z,u)=&\int_{\R^d} p(t,z,w)p(t,w,u)\,dw\\
\le& c_{2}H(z)T_t^V1(u)\\
\le& c_{3}H(z)\begin{cases} e^{-c_{4}tg(|u|)}, &\quad |u|\le \tilde s_0(t)\\
(1+|u|)^{-\alpha/2}, &\quad |u|> \tilde s_0(t),
\end{cases}
\end{split}
\end{equation}
where we have used \eqref{e:note2},  \eqref{l4-2-4}, \eqref{l4-2-3a} and \eqref{l4-2-3}.
 In the same way,  we can obtain that for every $c_5\ge2$,
$\tilde C_0h(|x|)< t \le \tilde C_0 T_0$,
 $|z|>\tilde s_0(t)$ and $|u|\le \tilde s_0(t)$,
\begin{equation}\label{e:lf2--}p(c_5t,z,u)=\int_{\R^d} p((c_5-1)t,z,w)p(t,w,u)\,dw \le c_{6}H(z) e^{-c_{4}tg(|u|)}.\end{equation}

For any
$\tilde C_0h(|x|)< t \le \tilde C_0 T_0$
and $z,u\in \R^d$ with $|z|\ge \tilde s_0(t)$ and $|u|\ge \tilde s_0(t)$,
we write
\begin{align*}
  p(4t,z,u)
&=\int_{\R^d} p(2t,z,w)p(2t,w,u)\,dw\\
&= \int_{\{|w|\le \tilde s_0(t)\}}+\int_{\{\tilde s_0(t)< |w|\le |u|\}}+\int_{\{|w|> |u|\}}  p(2t,z,w)p(2t,w,u)\,dw\\
&= :I_1+I_2+I_3. \end{align*}

According to \eqref{e:lf2},
$$p(2t,z,w)\le c_{3}H(z)e^{-c_{4}t g(|w|)},\quad p(2t,u,w)\le c_{3}H(u)e^{-c_{4}tg(|w|)}$$ for all $|w|\le \tilde s_0(t)$, and so
$$I_1\le c_{7} H(z)H(u)\int_{\{|w|\le \tilde s_0(t)\}}e^{-2c_{14}tg(|w|)}\,dw.$$
Similarly, by \eqref{e:lf2}, we have
\begin{align*}
I_2&\le  c_{8} H(z)H(u)\int_{\{\tilde s_0(t)< |w|\le |u|\}}(1+|w|)^{-\alpha}\,dw\\
&\le c_{9} H(z)H(u)\max\{(1+|u|)^{d-\alpha},1\}(1+\log(1+|u|)\I_{\{d=\alpha=1\}}).
\end{align*}
Furthermore, by \eqref{e:lf2} again, for all $|w|\ge |u|\ge \tilde s_0(t)$,
$$p(2t,z,w)\le c_{3}H(z)(1+|w|)^{-\alpha/2},\quad p(2t,w,u)\le c_{3}H(w)(1+|u|)^{-\alpha/2}.$$
This yields (by noting that $H(w)=\frac{1}{g(|w|)(1+|w|)^{d+\alpha}}$)
\begin{align*}
I_3\le & c_{10}H(z)(1+|u|)^{-\alpha/2} \int_{\{|w|> |u|\}}(1+|w|)^{-\alpha/2}H(w)\,dw\\
\le& c_{11}H(z)(1+|u|)^{-2\alpha}g(|u|)^{-1}\le  c_{12}H(z)H(u)(1+|u|)^{d-\alpha}.
\end{align*}

Combining with all the estimates above for $I_1$, $I_2$ and $I_3$, we get
\begin{equation}\label{e:lf3}
\begin{split}
 p(4t,z,u)\le &c_{13}H(z)H(u) \max\{(1+|u|)^{d-\alpha},1\}(1+\log(1+|u|)\I_{\{d=\alpha=1\}})\\
&\times\int_{\{|w|\le \tilde s_0(t)\}}e^{-c_{14}tg(|w|)}\,dw
\end{split}
\end{equation}
for all $z,u\in \R^d$, $t>0$ with $|z|\ge \tilde s_0(t)$, $|u|\ge \tilde s_0(t)$ and
$\tilde C_0h(|x|)< t \le \tilde C_0 T_0$.
When $d\neq \alpha$, this along with
 \eqref{e:lf2--} yields that for all
$\tilde C_0h(|x|)< t \le \tilde C_0 T_0$
and $|z|\ge \tilde s_0(t)$,
\begin{equation}\label{e:lf4}
\begin{split}
  p(4t,z,u)
 \le c_{15}H(z)
\begin{cases} e^{-c_{16}t g(|u|)},&\quad |u|\le \tilde s_0(t),\\
(1+|u|)^{-\min\{d+\alpha,2\alpha\}}\displaystyle\int_{\{|w|\le \tilde s_0(t)\}}e^{-c_{17}tg(|w|)}\,dw,&\quad |u|> \tilde s_0(t).
\end{cases}
\end{split}
\end{equation}
Here we have used the fact that
\begin{align*}
\int_{\{|w|\le \tilde s_0(t)\}}e^{-c_{17}tg(|w|)}\,dw&\ge
\int_{\{|w|\le   c_{18}\}}e^{-c_{17}tg(|w|)}\,dw\ge c_{19}.
\end{align*}
When $ d= \alpha=1$, we can use the following estimate instead of \eqref{l4-2-4}:
$$T_{t}^V1(z)\le c_{20}(1+|z|)^{-\alpha}\log(1+|z|),\quad z\in \R^d,\ t>0\ {\rm with}\ t\ge\tilde C_0h(|z|),$$ thanks to \eqref{l4-2-3a-}. Then, by the same arguments as above, we can prove that \eqref{e:lf4} still holds.

In the next step, we obtain the estimate of $p(6t,z,u)$ for every
$\tilde C_0h(|x|)< t \le \tilde C_0 T_0$
and $z,u\in \R^d$ with $|z|> \tilde s_0(t)$ and $|u|> \tilde s_0(t)$
via the iterated estimate \eqref{e:lf4}.
We have
\begin{align*}
p(6t,z,u) & = \int_{\R^d} p(4t,z,w)p(2t,w,u)\,dw \\
&=\int_{\{|w|\le \tilde s_0(t)\}}+\int_{\{\tilde s_0(t)<|w|\le |u|\}}+\int_{\{|w|> |u|\}} p(4t,z,w)p(2t,w,u)\,dw\\
&=:J_1+J_2+J_3.
\end{align*}
Applying \eqref{e:lf2} and \eqref{e:lf4} directly, we derive
\begin{align*}
J_1\le c_{21}H(z)H(u)\int_{\{|w|\le \tilde s_0(t)\}}e^{-c_{22}tg(|w|)}\,dw.
\end{align*}
Using  \eqref{e:lf2} and the second inequality of \eqref{e:lf4} instead of the first one for the estimate of $p(4t,z,w)$, we get
\begin{align*}
J_2&\le c_{23}H(z)H(u)\int_{\{|w|\le \tilde s_0(t)\}}e^{-c_{24}tg(|w|)}\,dw\cdot
\int_{\{\tilde s_0(t)<|w|\le |u|\}}(1+|w|)^{-\min\{d+\alpha,2\alpha\}-\alpha/2}\,dw\\
&\le c_{25}H(z)H(u)\max\{(1+|u|)^{d-2\alpha},1\}\cdot\int_{\{|w|\le \tilde s_0(t)\}}e^{-c_{24}tg(|w|)}\,dw
\end{align*} and
\begin{align*}
J_3&\le c_{26}H(z)(1+|u|)^{-\min\{d+\alpha,2\alpha\}}\int_{\{|w|\le \tilde s_0(t)\}}e^{-c_{27}tg(|w|)}\,dw\cdot
\int_{\{|w|> |u|\}}(1+|w|)^{-\alpha/2}H(w)\,dw\\
&\le c_{28}H(z)H(u)\max\{(1+|u|)^{d-2\alpha},1\}\cdot\int_{\{|w|\le \tilde s_0(t)\}}e^{-c_{27}tg(|w|)}\,dw.
\end{align*}
Combining with all the estimates above for $J_1$--$J_3$ yields that for all $\tilde C_0 T_{\infty} < t \le \tilde C_0 T_0$ and $z,u\in \R^d$ with $|z|> \tilde s_0(t)$ and $|u|> \tilde s_0(t)$,
\begin{equation}\label{l4-2-6}\begin{split}
p(6t,z,u) \le c_{29}H(z)H(u)\max\{(1+|u|)^{d-2\alpha},1\}\cdot\int_{\{|w|\le \tilde s_0(t)\}}e^{-c_{30}tg(|w|)}\,dw.
\end{split}\end{equation}

Furthermore, according to the decomposition
\begin{align*}
p(8t,z,u)
&=\int_{\{|w|\le \tilde s_0(t)\}}+\int_{\{\tilde s_0(t)<|w|\le |u|\}}+\int_{\{|w|> |u|\}} p(6t,z,w)p(2t,w,u)\,dw,
\end{align*}
one can apply \eqref{l4-2-6} for the estimate of $p(6t,z,w)$ and follow the arguments for \eqref{l4-2-6} to prove that
for every
$\tilde C_0 h_0(|x|) < t \le \tilde C_0 T_0$,
$|z|> \tilde s_0(t)$ and $|u|> \tilde s_0(t)$,
\begin{align*}
p(8t,z,u)\le c_{31}H(z)H(u)\max\{(1+|u|)^{d-3\alpha},1\}\cdot\int_{\{|w|\le \tilde s_0(t)\}}e^{-c_{32}tg(|w|)}\,dw.
\end{align*}
Then, following this procedure iteratively for $N_2$ times, where $N_2$ is such that
$d-N_2\alpha<0$, we can show that for all
$\tilde C_0 h_0(|x|) < t \le \tilde C_0 T_0$,
$|z|> \tilde s_0(t)$ and $|u|> \tilde s_0(t)$,
\begin{align*}
p(2(N_2+1)t,z,u)&\le c_{33}H(z)H(u)\int_{\{|w|\le \tilde s_0(t)\}}e^{-c_{34}tg(|w|)}\,dw \\
& =  c_{33}H(z)H(u)\int_{ \{|w|\le s_0(\tilde C_0^{-1}t)\}}e^{-c_{34}tg(|w|)}\,dw.
\end{align*}
Combining this estimate with the fact that
$\tilde C_0 h(|x|)< t \le \tilde C_0T_0$ implies $|y|\ge |x|> \tilde s_0(t)$,
we prove the desired bound for $2(N_2+1) \tilde C_0 h(|x|) < t \le 2(N_2+1) \tilde C_0  T_0$, i.e.,\ \eqref{l4-2-2} holds with $C_0 =2(N_2+1) \tilde C_0 $
and $C_5=\tilde C_0^{-1}$.
\end{proof}

\begin{lemma}\label{l4-7}
For sufficiently large $C_0>0$ there are constants $C_{7}, C_{8}>0$ such that for all $t> C_0 T_0$,
\begin{equation}\label{l4-7-0}
C_{7}e^{-\lambda_1 t}H(x)H(y)\le p(t,x,y)\le C_{8}e^{-\lambda_1 t}H(x)H(y).
\end{equation}
\end{lemma}
\begin{proof}
For every positive constants $c_1,c_2$,
$$
c_3\le \int_{\{|z|\le  2\}}e^{-c_1g(|z|)}\,dz\le \int_{\{|z|\le s_0(c_2)\}}e^{-c_1g(|z|)}\,dz\le  c_4.
$$
Then, according to \eqref{l4-1-1}, \eqref{l4-2-2} and Proposition \ref{P:2.4}, it holds that
\begin{equation}\label{l4-7-1}
c_5 H(z)H(u)\le  p( C_0 T_0,z,u)  \le c_{6}H(z)H(u),\quad z,u\in \R^d.
\end{equation} Indeed, when $h(\min\{|z|,|u|\}) < T_0$ and $C_0$ is large enough, \eqref{l4-7-1} follows from \eqref{l4-1-1} and \eqref{l4-2-2}  by taking $t = C_0 T_0$. On the other hand, if $h(\min\{|z|,|u|\}) \ge T_0$, then $\min\{|z|,|u|\} \le 2$,
and so $C_0T_0\le c_7t_0(\min\{|z|,|u|\})$. Thus, by Proposition \ref{P:2.4}, we have
\begin{align*}
p( C_0 T_0,z,u)\asymp &\left(\frac{1}{C_0T_0g(\max\{|u|,|z|\})}\wedge 1\right) e^{-c_{7}C_0T_0g(\min\{|u|,|z|\})}\left((C_0T_0)^{-d/\alpha}\wedge \frac{C_0 T_0}{|u-z|^{d+\alpha}}\right)\\
\asymp&\frac{1}{g(|u|)} \frac{1}{g(|z|)}
\frac{1}{(1+|u|)^{d+\alpha}} \frac{1}{(1+|z|)^{d+\alpha}}= H(u) H(z).
\end{align*}
Here the second step is due to the facts that for every $z,u\in \R^d$ with $\min\{|z|,|u|\}\le 2$,
\begin{align*}
&g(\max\{|z|,|u|\})\asymp g(|z|)g(|u|),\quad g(\min\{|z|,|u|\})\asymp 1,\\
&\left((C_0T_0)^{-d/\alpha}\wedge \frac{C_0 T_0}{|u-z|^{d+\alpha}}\right)
\asymp (1+|u|)^{-d-\alpha}(1+|z|)^{-d-\alpha}.
\end{align*}
This implies that \eqref{l4-7-1} is still true when $h(\min\{|z|,|u|\}) \ge T_0$.

Therefore, by \eqref{l4-6-2} and \eqref{l4-7-1}, for every $ t>C_0 T_0$,
\begin{align*}
p(t,x,y)&=\int_{\R^d}p(t-C_0 T_0,x,w)p(C_0T_0,w,y)\,dw\\
&\le c_{6}H(y)\int_{\R^d}p(t-C_0T_0,x,w)H(w)\,dw\\
&=c_{6}H(y)T_{t-C_0T_0}^VH(x)\le c_{8}e^{-\lambda_1 t}H(x)H(y),
\end{align*}
where in the last inequality we have used \eqref{l4-6-2}.

Applying \eqref{l4-6-2} and \eqref{l4-7-1} again, we derive that for every $t>C_0 T_0$,
\begin{align*}
p(t,x,y)&=\int_{\R^d}p(t-C_0T_0,x,w)p(C_0T_0,w,y)\,dw\\
&\ge c_{5}H(y)\int_{\R^d}p(t-C_0T_0,x,w)H(w)\,dw\\
&=c_{5}H(y)T_{t-C_0T_0}^VH(x)\ge c_{9}e^{-\lambda_1 t}H(x)H(y).
\end{align*}
Then, we prove \eqref{l4-7-0}.
\end{proof}

\subsection{$t_0(\cdot)$ is almost increasing}
In this subsection, we will assume that $t_0(\cdot)$ is almost increasing.
In particular,
$
c_{t_0}:=\inf_{z\in \R^d}t_0(|z|)>0.
$
Set
$$
s_1(t):=\inf\{s>0: h(s)\ge t\}\vee 2,\quad\ t>0,
$$
where $h(\cdot)$ is a strictly increasing and continuous function given in \eqref{e4-1} and we use the convention that
$\inf \emptyset =-\infty$.  Recall that $T_0=h(2)>0$ and $T_{\infty}=h(\infty)=\lim_{s \to \infty} h(s) \in (0,\infty]$.  In this part, we are looking for estimates of $p(t,x,y)$ when $t>C_0t_0(|x|)$. Thus, by taking $C_0$ large enough, we only need to consider the case that $t>T_0$.
In particular, in this case we have $s_1(t)= \inf\{s>0: h(s)\ge t\}$.

\begin{lemma}\label{l4-4}
There exist constants $C_0,C_{9}>0$ such that for every $z,u\in \R^d$, $t>0$ with $t\ge C_0t_0(|z|)$,
\begin{equation}\label{l4-4-1}
p(t,z,u)\le C_{9} H(z)H(u).
\end{equation}
\end{lemma}
\begin{proof}
We first note that when $t_0(\cdot)$ is a bounded function, i.e.,\ $T_{\infty}=h(\infty) < \infty$,
$t_0(\cdot)$ is almost increasing and almost decreasing at the same time. Indeed, we can simply modify the profile $h(\cdot)$ to be a continuous and decreasing function and keep the value $h(2)$ at the cost of comparability constants in \eqref{e4-1}. In this case, the estimate \eqref{l4-4-1} follows directly from the
upper bound estimate
in Lemma \ref{l4-7}.
Therefore, for the rest of the proof we assume that $t_0(\cdot)$ is an almost increasing function and the corresponding profile $h$
satisfies that
$T_{\infty} = \infty$.

According to \eqref{e4-1},
\eqref{l4-2-3a-} and \eqref{l4-5-1},
there is a constant $\tilde C_0>0$ so that
\begin{equation}\label{l4-4-3}
T_{t}^V1(z)\le c_1(1+|z|)^{-\alpha/2},\quad z\in \R^d,\ t>0\ {\rm with}\ t\ge\tilde C_0h(|z|),
\end{equation} and
\begin{equation}\label{l4-4-4}
p(t,z,u)\le c_2H(z),\quad u,z\in \R^d,\ t>0\ {\rm with}\ t\ge \tilde C_0h(|z|).
\end{equation}
Without loss of generality  in the following we can assume that $\tilde C_0>T_0/c_{t_0}=h(2)/c_{t_0}$.
On the other hand, by \eqref{e4-1}, the definition of $s_1(\cdot)$ and the fact that $h(\cdot)$ is
continuous and strictly increasing,
we know that there exists  $\tilde C_1>0$ large enough such that for all $t\ge \tilde C_1 T_0 $,
\begin{equation}\label{l4-4-2}
\begin{split}
t\ge \tilde C_1h(|z|/6)\ge \tilde C_0h(|z|),\ &\quad z\in \R^d\,\, {\rm with}\ |z|\le 6s_1\left(\frac{t}{\tilde C_1}\right)=:6\tilde s_1(t),\\
t\le \tilde C_1h(|z|/6)\le \tilde C_1h(|z|),\ &\quad z\in \R^d\,\, {\rm with}\ |z|>6s_1\left(\frac{t}{\tilde C_1}\right)=:6\tilde s_1(t).
\end{split}
\end{equation}
The proof below is split into
three steps.

\noindent
{\bf Step 1}:\
According to \eqref{l4-4-3}, \eqref{l4-4-4} and \eqref{l4-4-2}, we obtain that for all $z,u\in \R^d$ and $t>\tilde C_1 T_0$ with
$|z|\le \tilde s_1(t)$ and $|u|\le 6\tilde s_1(t)$,
\begin{equation}\label{e:ppffoo}
 p(2t,z,u)   =\int_{\R^d} p(t,z,w)p(t,w,u)\,dw   \le c_{3}H(z)T_{t}^V1(u) \le c_{4}H(z)(1+|u|)^{-\alpha/2}.
\end{equation}
When $ t>\tilde C_1 T_0$, $|z|\le \tilde s_1(t)$ and $|u|>\tilde s_1(t)$ with $|u|>6|z|$, set
\begin{align*}
p(2t,z,u)&=\left(\int_{\{|z-w|\le {|u|}/{3}\}}+\int_{\{|z-w|>{|u|}/{3}\}}\right)
p(t,z,w)p(t,w,u)\,dw=:I_1+I_2.
\end{align*}
If $|z-w|\le {|u|}/{3}$, then $|u-w|\ge |u-z|-|z-w|\ge |u|-|z|-|u|/3\ge |u|/2\ge c_{5}t^{1/\alpha}$,
where in the last inequality we used the fact that for $t\ge \tilde C_1T_0$,
$$|u|\ge \tilde s_1(t)\ge e^{c_{6}(1+t)}\ge c_7t^{1/\alpha},$$
thanks to the fact that $h(s)\le c_{8}\log(2+s)$
(which can be verified by \eqref{l4-2-1b} directly).
By \eqref{l2-5-1} and \eqref{l4-4-3}, we get
\begin{align*}
I_1&\le \int_{\{|z-w|\le {|u|}/{3}\}}\left(\frac{c_{9}}{t\max\{g(|u|),g(|w|)\}}\wedge 1\right)\frac{t}{|u-w|^{d+\alpha}}
p(t,z,w)\,dw\\
&\le \frac{c_{10}}{g(|u|)(1+|u|)^{d+\alpha}}T^V_{t}1(z)\le
c_{11}H(u)
(1+|z|)^{-\alpha/2}.
\end{align*}
Similarly, using \eqref{l2-5-1} again we have
\begin{align*}
I_2&\le \int_{\{|z-w|> {|u|}/{3}\}}\left(\frac{c_{12}}{t\max\{g(|z|),g(|w|)\}}\wedge 1\right)\frac{t}{|z-w|^{d+\alpha}}
p(t,w,u)\,dw\\
&\le \frac{c_{13}}{g(|u|)(1+|u|)^{d+\alpha}}T^V_{t}1(u) \\
&\le \frac{c_{14}}{g(|u|)(1+|u|)^{d+\alpha}}\left(e^{-c_{15}tg(|u|)}+\frac{t}{(1+|u|)^{\alpha}}\right)\\
&\le c_{14}H(u)\left(e^{-c_{15}tg(|z|)}+\frac{t}{(1+|z|)^{\alpha}}\right)\le c_{16} H(u)(1+|z|)^{-\alpha/2},
\end{align*}
where the second inequality follows from the fact
that $g(|w|)\ge g(|z-w|-|z|)\ge g(|u|/6)\ge c_{17}g(|u|)$ for all
$z,w\in \R^d$ with $|z-w|>{|u|}/{3}$ and $|z|\le |u|/6$,
the third inequality is due to \eqref{est1},
the fourth inequality follows from the fact $|u|\ge 6|z|$, and in the
last inequality we have used the arguments from the proof of
\eqref{c2-2-2}
(by taking $\tilde C_1$ larger if necessary).
Hence, by all above estimates, for every $|z|\le \tilde s_1(t)/6$,
\begin{equation}\label{l4-4-6}
\begin{split}
p(2t,z,u)&\le
\begin{cases}
c_{18}H(z)(1+|u|)^{-\alpha/2},\ &|u|\le \tilde s_1(t),\\
c_{18}H(u)(1+|z|)^{-\alpha/2},\ &|u|>\tilde s_1(t),
\end{cases}
\\
&=
\begin{cases}
c_{18}H(z)H(u)(1+|u|)^{d+\alpha/2}g(|u|),\ &|u|\le \tilde s_1(t),\\
c_{18}H(z)H(u)(1+|z|)^{d+\alpha/2}g(|z|),\ &|u|>\tilde s_1(t);
\end{cases}
\end{split}
\end{equation} and for every $\tilde s_1(t)/6 \le |z|\le \tilde s_1(t)$,
\begin{equation}\label{l4-4-6---}
\begin{split}
p(2t,z,u)&\le
\begin{cases}
c_{19}H(z)H(u)(1+|u|)^{d+\alpha/2}g(|u|),\ &|u|\le 6\tilde s_1(t),\\
c_{19}H(z)H(u)(1+|z|)^{d+\alpha/2}g(|z|),\ &|u|>6\tilde s_1(t),
\end{cases}
\\
&\le
\begin{cases}
c_{20}H(z)H(u)(1+|u|)^{d+\alpha/2}g(|u|),\ &|u|\le \tilde s_1(t),\\
c_{20}H(z)H(u)(1+|z|)^{d+\alpha/2}g(|z|),\ &|u|>\tilde s_1(t),
\end{cases}
\end{split}
\end{equation}
where the last inequality we have used the fact that $|u|\asymp |z|$ and
$H(u)\asymp H(z)$ for every $\tilde s_1(t)/6 \le |z|\le \tilde s_1(t)$ and $\tilde s_1(t) \le |u|\le 6\tilde s_1(t)$.
Then, we show that the inequality \eqref{l4-4-6} holds for every $|z|\le \tilde s_1(t)$.

Below we assume that $t>\tilde C_1T_0$ and consider the case that $|z|>\tilde s_1(t)$.
According to \eqref{l4-4-6} (by noting that \eqref{l4-4-6} has been verified for every $|z|\le \tilde s_1(t)$ and exchanging the position of
$z$ and $u$),
\begin{align}\label{l4-4-6--}
p(2t,z,u)\le c_{21}H(z)(1+|u|)^{-\alpha/2},\quad |z|>\tilde s_1(t),\ |u|\le \tilde s_1(t).
\end{align}
When $|z|>\tilde s_1(t)$ and $|u|>\tilde s_1(t)$, by \eqref{l4-4-2} and its proof, there is a constant $c_{22}>0$ such that
$t\le c_{22}\min\{t_0(|z|,t_0(|u|)\}$. Thus, it follows from \eqref{p3-1-1} that
\begin{align*}
p(2t,z,u)\le c_{23}\left(\frac{1}{t\max\{g(|z|),g(|u|)\}}\wedge 1\right)\left(t^{-d/\alpha}\wedge \frac{t}{|z-u|^{d+\alpha}}\right)e^{-c_{21}t\min\{g(|z|),g(|u|)\}}.
\end{align*}
Hence, for every $t>\tilde C_1T_0$ and  $|z|>\tilde s_1(t)$,
\begin{equation}\label{l4-4-7}\begin{split}
& p(2t,z,u)  \\ & \le
\begin{cases}
c_{21}H(z)(1+|u|)^{-\alpha/2},\ &|u|\le \tilde s_1(t),\\
c_{23}e^{-c_{21}t\min\{g(|z|),g(|u|)\}}\left(\frac{1}{t\max\{g(|z|),g(|u|)\}}\wedge 1\right)\left(t^{-d/\alpha}\wedge \frac{t}{|z-u|^{d+\alpha}}\right),\ &|u|>\tilde s_1(t).
\end{cases}
\end{split}\end{equation}

{\bf Step 2}:\ In what follows, we will make use of \eqref{l4-4-6} and \eqref{l4-4-7} to obtain some
improved estimates for $p(t,z,u)$, which will be applied to the iteration procedure later on. We always assume that $t>\tilde C_1T_0$.

We first assume that $d\neq \alpha$.
For every $|z|\le \tilde s_1(t)$ and $|u|\le \tilde s_1(t)$, let
\begin{align*}
p(4t,z,u)
&=\left(\int_{\{|w|\le |u|\}}+\int_{\{|u|\le |w|\le 6\tilde s_1(t)\}}
\int_{\{|w|>6\tilde s_1(t)\}}\right)p(2t,z,w)p(2t,w,u)\,dw\\
&=:J_1+J_2+J_3.
\end{align*}
By \eqref{l4-4-6} (which has been proven for every $|z|\le \tilde s_1(t)$),
\begin{align*}
J_1 &\le c_{24} H(z)H(u)\int_{\{|w|\le |u|\}}H^2(w)(1+|w|)^{2(d+\alpha/2)}g^2(|w|)\,dw\\
&\le c_{25}H(z)H(u)\max\{(1+|u|)^{d-\alpha},1\}.
\end{align*} According to \eqref{e:ppffoo} and  the proof of \eqref{l4-4-6},
\begin{align*}
J_2 & \le c_{26}H(z)(1+|u|)^{-\alpha/2}\int_{\{|u|\le |w|\le 6\tilde s_1(t)\}}(1+|w|)^{-\alpha/2}H(w)\,dw \\ & \le
c_{27}H(z)H(u)(1+|u|)^{d-\alpha}.
\end{align*}
On the other hand, $|w|\ge 6\max\{|z|,|u|\}$ if $|w|>6\tilde s_1(t)$. Applying
\eqref{l4-4-6} again, we derive
\begin{align*}
J_3 &\le c_{28}(1+|z|)^{-\alpha/2}(1+|u|)^{-\alpha/2}\int_{\{|w|>6\max\{|z|,|u|\}\}}H^2(w)\,dw \\
&
\le  c_{28}(1+|z|)^{-\alpha/2}(1+|u|)^{-\alpha/2}  (g(|z|)g(|u|))^{-1}  \int_{\{|w|>6\max\{|z|,|u|\}\}}(1+|w|)^{-2(d+\alpha)}\,dw\\
& \le c_{29} H(z)H(u)(1+|u|)^{d-\alpha}.
\end{align*}
Hence, combining all above estimates together yields that
\begin{equation}\label{l4-4-8}
p(4t,z,u)\le c_{30}H(z)H(u)\max\{(1+|u|)^{d-\alpha},1\},\quad t> \tilde C_1T_0,
|z|\le \tilde s_1(t), |u|\le \tilde s_1(t).
\end{equation}
When $d= \alpha$, we can use the following estimate instead of \eqref{l4-4-3}:
$$T_{t}^V1(z)\le c_{31}(1+|z|)^{-\alpha}\log(1+|z|),\quad z\in \R^d,\ t>0\ {\rm with}\ t\ge\tilde C_0h(|z|),$$ thanks to \eqref{l4-2-3a-}. Then, by the same arguments as above, we can prove that \eqref{l4-4-8} still holds true.

On the other hand, for every $|z|\le \tilde s_1(t)$ and $|u|> \tilde s_1(t)$, set
\begin{align*}
p(6t,z,u)&=\left(\int_{\{|w|\le \tilde s_1(t)\}}+
\int_{\{|w|>\tilde s_1(t)\}}\right)
p(4t,z,w)p(2t,w,u)\,dw\\
&=:K_1+K_2.
\end{align*}
Applying  \eqref{l4-4-8} and \eqref{l4-4-6--} to  $p(4t,z,w)$ and $p(2t,w,u)$ respectively, we have
\begin{align*}
K_1&\le c_{32}H(z)H(u)\max\{(1+|z|)^{d-\alpha},1\}\int_{\{|w|\le \tilde s_1(t)\}}(1+|w|)^{-\alpha/2}H(w)\,dw\\
&\le c_{33}H(z)H(u)\max\{(1+|z|)^{d-\alpha},1\}.
\end{align*}
Meanwhile, by
\eqref{l4-4-2},
\eqref{l4-4-7} and
both of the second inequalities in \eqref{l4-4-6} and \eqref{l4-4-6---}, we have
\begin{align*}
K_2&\le \frac{c_{34}}{(1+|z|)^{\alpha/2}}\int_{\{|w|>\tilde s_1(t)\}}
\frac{e^{-c_{35}t\min\{g(|w|),g(|u|)\}}H(w)}{t\max\{g(|w|),g(|u|)\}}\left(t^{-d/\alpha}\wedge \frac{t}{|w-u|^{d+\alpha}}\right)\,dw \\
&\le \frac{c_{34}e^{-c_{35}tg(|z|)}}{g(|u|)(1+|z|)^{\alpha/2}}\int_{\{|w|>\tilde s_1(t)\}}H(w)t^{-1}\left(t^{-d/\alpha}\wedge \frac{t}{|w-u|^{d+\alpha}}\right)\,dw \\
&=\frac{c_{34}e^{-c_{35}tg(|z|)}}{g(|u|)(1+|z|)^{\alpha/2}}\cdot\left[\sum_{i=1}^3\int_{\{|w|>\tilde s_1(t)\}\cap A_i}H(w)t^{-1}\left(t^{-d/\alpha}\wedge \frac{t}{|w-u|^{d+\alpha}}\right)\,dw\right]\\
&=:\frac{c_{34}e^{-c_{35}tg(|z|)}}{g(|u|)(1+|z|)^{\alpha/2}}\left( K_{21}+K_{22}+K_{23}\right),
\end{align*}
where $A_1:=\{w\in \R^d: |w-u|\le c_{36}t^{1/\alpha}\}$, $A_2:=\{w\in \R^d: c_{36}t^{1/\alpha}<|w-u|\le |u|/2\}$ and
$A_3:=\{w\in \R^d: |w-u|>|u|/2\}$. Here we used again the fact that
$$|u|\ge \tilde s_1(t)\ge e^{c_{37}(1+t)}\ge 8c_{36}t^{1/\alpha},$$
which can be verified by the property that $h(s)\le c_{38}\log(2+s)$.

By \eqref{l4-4-2}, we know that $t\ge \tilde C_0h(|z|)$ when $|z|\le \tilde s_1(t)$.
Increasing the constant  $\tilde C_0$  if necessary and following the argument for \eqref{c2-2-2}, we obtain
\begin{equation}\label{l4-4-7a}
\begin{split}
e^{-c_{35}tg(|z|)}&\le e^{-c_{35}\tilde C_0h_0(|z|)g(|z|)}\le
c_{39}g(|z|)^{-1}(1+|z|)^{-d-\alpha/2},\quad |z|\le \tilde s_1(t),
\end{split}
\end{equation}
where in the last step we have used the lower bound of $g(\cdot)$ in \eqref{l4-2-1b}.
Furthermore, by direct computations we can find that
$$
K_{21}\le  \sup_{w\in A_1}H(w)\cdot\left(\int_{\{|w-u|\le c_{36}t^{1/\alpha}\}}t^{-d/\alpha-1}\,dw\right)\le c_{40}H(u),
$$
$$
K_{22}\le  \sup_{w\in A_2}H(w)\cdot\left(\int_{\{c_{36}t^{1/\alpha}<|w-u|\le |u|/2\}}\frac{1}{|w-u|^{d+\alpha}}\,dw\right)\le c_{40}H(u),
$$ and
$$
K_{23}\le \frac{c_{40}}{(1+|u|)^{d+\alpha}}\int_{\{|w-u|>|u|/2\}}H(w)\,dw\le
c_{41}H(u).
$$
Putting all the estimates above together yields
\begin{equation*}\label{l4-4-9a}
K_2\le c_{42}H(z)H(u).
\end{equation*}
By combining this bound with the estimate of $K_1$ above,  we obtain that
\begin{equation*}
p(6t,z,u) \le c_{43} H(z)H(u)\max\{(1+|z|)^{d-\alpha},1\},\quad |z|\le \tilde s_1(t), |u|> \tilde s_1(t).
\end{equation*}
This, along with \eqref{l4-4-8}
(for the case that $|z|\le \tilde s_1(t)$ and $|u|\le \tilde s_1(t)$),
 yields that for every $t>\tilde C_1T_0$ and $|z|\le \tilde s_1(t)$,
\begin{equation}\label{l4-4-9}
\begin{split}
p(6t,z,u)&\le
\begin{cases}
c_{44}H(z)H(u)\max\{(1+|u|)^{d-\alpha},1\},\ &|u|\le \tilde s_1(t),\\
c_{44}H(z)H(u)\max\{(1+|z|)^{d-\alpha},1\},\ &|u|>\tilde s_1(t).
\end{cases}
\end{split}
\end{equation} Indeed, by \eqref{l4-4-8},
$$
 p(6t,z,u) =p\left(4\times \frac{3t}{2},z,u\right)\le c_{44}H(z)H(u)\max\{(1+|u|)^{d-\alpha},1\}
$$ when $ |u|\le \tilde s_1(3t/2)$ and $|z|\le \tilde s_1(3t/2).$
According to this and the fact that $\tilde s_1(3t/2)\ge \tilde s_1(t)$, we obtain the desired assertion in \eqref{l4-4-9}.

{\bf Step 3:}\ Now we will apply the iteration procedure and use the improved estimate \eqref{l4-4-9}. In fact, for every $t>\tilde C_1T_0$, $|z|\le \tilde s_1(t)$ and $|u|\le \tilde s_1(t)$, we set
\begin{align*}
p(8t,z,u)
&=\left(\int_{\{|w|\le |u|\}}+\int_{\{|u|\le |w|\le 6\tilde s_1(t)\}}
\int_{\{|w|>6\tilde s_1(t)\}}\right)p(6t,z,w)p(2t,w,u)\,dw\\
&=:L_1+L_2+L_3.
\end{align*}
Using \eqref{l4-4-6}, \eqref{l4-4-6---} and \eqref{l4-4-9} (for the estimate of $p(6t,z,w)$),
we obtain
\begin{align*}
L_1&\le c_{45}H(z)H(u)\int_{\{|w|\le |u|\}}H^2(w)g(|w|)\max\{(1+|w|)^{2d-\alpha/2},(1+|w|)^{d+\alpha/2}\}\,dw\\
&\le c_{46}H(z)H(u)\max\{(1+|u|)^{d-2\alpha},1\},
\end{align*}
\begin{align*}
L_2&\le c_{47}H(z)(1+|u|)^{-\alpha/2}\int_{\{|u|\le |w|\le 6\tilde s_1(t)\}}H^2(w)\max\{(1+|w|)^{d-\alpha},1\}\,dw\\
&\le c_{48}H(z)H(u)(1+|u|)^{d-2\alpha}
\end{align*} and
\begin{align*}
L_3&
\le c_{49}H(z)\max\{(1+|z|)^{d-\alpha},1\}(1+|u|)^{-\alpha/2}\int_{\{|w|>\max\{|z|,|u|\}\}}H^2(w)\,dw\\
&\le c_{50}H(z)H(u)(1+|u|)^{d-2\alpha}.
\end{align*}
Combining with all above estimates yields that
\begin{equation}\label{l4-4-10}
p(8t,z,u)\le c_{51}H(z)H(u)\max\{(1+|u|)^{d-2\alpha},1\},\quad
|z|\le \tilde s_1(t),\ |u|\le \tilde s_1(t).
\end{equation}

On the other hand, when $|z|\le \tilde s_1(t)$ and $|u|>\tilde s_1(t)$,
\begin{align*}
p(10t,z,u)&=\left(\int_{\{|w|\le \tilde s_1(t)\}}+
\int_{\{|w|>\tilde s_1(t)\}}\right)p(8t,z,w)p(2t,w,u)\,dw\\
&=:M_1+M_2.
\end{align*}
According to \eqref{l4-4-10} for  $p(8t,z,w)$ and \eqref{l4-4-6} for $p(2t,w,u)$ respectively, we can verify directly that
\begin{align*}
M_1&\le c_{52}H(z)H(u)\max\{(1+|z|)^{d-2\alpha},1\}\int_{\{|w|\le \tilde s_1(t)\}}(1+|w|)^{-\alpha/2}H(w)\,dw\\
&\le c_{53}H(z)H(u)\max\{(1+|z|)^{d-2\alpha},1\}.
\end{align*}
Applying the arguments for the estimate of $K_2$ in {\bf Step 2} (by choosing $\tilde C_0$ large enough if
necessary), we can derive
\begin{align*}
M_2&\le c_{54}H(z)H(u).
\end{align*}
Thus, for all $|z|\le \tilde s_1(t)$ and $|u|>\tilde s_1(t)$,
\begin{align*}
p(10t,z,u)\le c_{55}H(z)H(u)\max\{(1+|z|)^{d-2\alpha},1\}.
\end{align*}
This, along with (the proof of) \eqref{l4-4-10}, yields that
for every $t>\tilde C_1T_0$ and $|z|\le \tilde s_1(t)$,
\begin{equation}\label{l4-4-11}
\begin{split}
p(10t,z,u)&\le
\begin{cases}
c_{56}H(z)H(u)\max\{(1+|u|)^{d-2\alpha},1\},\ &|u|\le \tilde s_1(t),\\
c_{56}H(z)H(u)\max\{(1+|z|)^{d-2\alpha},1\},\ &|u|>\tilde s_1(t).
\end{cases}
\end{split}
\end{equation}

Furthermore, by using the improved estimate \eqref{l4-4-11} and iterating the same arguments as above for $N_0$ times
with $N_0\alpha>d$, we can find that
$$
p\left((6+4N_0)t,z,u\right)\le c_{57}H(z)H(u),\quad z,u\in \R^d,\ t>\tilde C_1T_0  \ {\rm with}\ |z|\le \tilde s_1(t).
$$
We can now combine this estimate with the fact that the inequalities
$t \ge N_1C_*\tilde C_1 t_0(|x|) \ge N_1\tilde C_1 h(|x|)$
(with $N_1$ large enough to ensure $N_1h(0)>T_0$) imply that $|x|\le \tilde s_1(t)$.
This shows that the desired bound \eqref{l4-4-1} holds with
$C_0=N_1C_*\tilde C_1(6+4N_0)$
since $N_1$ can be chosen large enough so that $C_0 t_0(|x|)>\tilde C_1T_0$.
\end{proof}

Next, we can further improve the upper estimate for $p(t,x,y)$ when $t>C_0t_0(|x|)$.

\begin{lemma}\label{l4-6}
There exist constants $C_0,C_{10}>0$ such that for every $t>C_0t_0(|x|)$,
\begin{equation}\label{l4-6-0}
p(t,x,y)\le C_{10}e^{-\lambda_1 t}H(x)H(y).
\end{equation}
\end{lemma}
\begin{proof}
For every $z,u\in \R^d$ and $t>0$ with $t>C_0t_0(|z|)$,
\begin{equation}\label{l4-6-3}
\begin{split}
p(t,z,u)&=\int_{\R^d}p\left(t-C_0t_0(|z|),u,w\right)p\left(C_0t_0(|z|),w,z\right)\,dw\\
&\le c_1H(z)\int_{\R^d}p\left(t-C_0t_0(|z|),u,w\right)H(w)\,dw\\
&=c_1H(z)T_{t-C_0t_0(|z|)}^V H(u)\\
&\le c_2e^{-\lambda_1 t}e^{\lambda_1C_0t_0(|z|)}H(z)H(u).
\end{split}
\end{equation}
Here the constant $C_0$ is the same as that in Lemma \ref{l4-4}, the first inequality is due to
\eqref{l4-4-1}, and the last inequality follows from \eqref{l4-6-2}. By \eqref{l4-6-3}, we know that for every
$R_0>0$, \eqref{l4-6-0} holds true for all $t>C_0t_0(|x|)$ and $|x|\le R_0$.
Hence, it suffices
to show \eqref{l4-6-0} for any  $t>C_0t_0(|x|)$ and  $|x|>R_0$ with $R_0$ to be determined later.

For every $t>2C_0\max\{t_0(|x|),t_0(|y|)\}$,
\begin{equation}\label{l4-6-4}
\begin{split}
p(t,x,y)&=\int_{\R^d}p(t/2,x,w)p(t/2,w,y)\,dw \\
&\le c_3e^{-\lambda_1 t}H(x)H(y)\int_{\R^d}H^2(w)e^{c_4t_0(|w|)}\,dw\\
&\le c_5e^{-\lambda_1 t}H(x)H(y).
\end{split}
\end{equation}
Here the first inequality follows from \eqref{l4-6-3}, and in the last inequality we have used the fact
\begin{align}\label{l4-6-4a}
\int_{\R^d}H^2(w) e^{c_4t_0(|w|)}dw<+\infty,
\end{align}
which can be verified directly by \eqref{l4-2-1b}, thanks to the fact that $g(r)\to \infty$ as $r\to \infty$.
On the other hand,  similarly as in \eqref{l4-4-2}, we can find $\tilde C_1>0$  large enough such that for $t \geq \tilde C_1 T_0$,
\begin{equation}\label{l4-6-5}
\begin{split}
t\ge \tilde C_1h(|z|)\ge  2C_0t_0(|z|),\quad & z\in \R^d\,\, {\rm with}\ |z|\le s_1\left(\frac{t}{\tilde C_1}\right)=:\tilde s_1(t),\\
t\le \tilde C_1h(|z|)\le c_{6}t_0(|z|),\quad &  z\in \R^d\,\, {\rm with}\ |z|>s_1\left(\frac{t}{\tilde C_1}\right)=:\tilde s_1(t).
\end{split}
\end{equation}
Below we choose the constant $C_0$ large enough so that $C_0\inf_{x\in \R^d}t_0(|x|)\ge \tilde C_1 T_0$.
Then, according to
\eqref{l4-6-5} and \eqref{l4-6-3}, we have
\begin{equation}\label{l4-6-6}
p(t,z,u)\le c_{7}e^{-\lambda_1 t}e^{c_{8}t_0(|z|)}H(z)H(u),\quad
t \geq C_0t_0(|x|), \ |z|\le \tilde s_1(t);
\end{equation}
by \eqref{l4-6-5} and \eqref{l4-6-4}, we know that \eqref{l4-6-0} is true for every $|x|\le \tilde s_1(t)$ and $|y|\le \tilde s_1(t)$,
whenever $t> C_0t_0(|x|)$.

Next, it remains to consider the case that
$|x|\le \tilde s_1(t)$, $|y|> \tilde s_1(t)$ and $t\ge C_0t_0(|x|)$.
Set
\begin{align*}
p(2t,x,y)&=\left(\int_{\{|w|\le \tilde s_1(t)\}}+
\int_{\{|w|>\tilde s_1(t)\}}\right)
p(t,x,w)p(t,w,y)\,dw\\
&=:I_1+I_2.
\end{align*}
By \eqref{l4-6-6}, it holds that
\begin{align*}
I_1&\le c_{9}e^{-2\lambda_1 t}H(x)H(y)\int_{\{|w|\le \tilde s_1(t)\}}H^2(w)e^{c_{10}t_0(|w|)}\,dw\\
&\le c_{11}e^{-2\lambda_1 t}H(x)H(y),
\end{align*}
where the last inequality is due to \eqref{l4-6-4a}.

When $|w|>\tilde s_1(t)$ and $|y|>\tilde s_1(t)$, it follows from \eqref{l4-6-5} that there is a constant $c_{12}>0$ such that
$t\le c_{12}\min\{t_0(|w|),t_0(|y|)\}$. Hence, applying \eqref{p3-1-1}, we obtain
\begin{align*}
p(t,w,y)\le c_{13}\left(\frac{1}{t\max\{g(|w|),g(|y|)\}}\wedge 1\right)\left(t^{-d/\alpha}\wedge \frac{t}{|w-y|^{d+\alpha}}\right)e^{-c_{14}t\min\{g(|w|),g(|y|)\}}.
\end{align*}
By \eqref{l4-6-6} and the fact that $|x|\le \tilde s_1(t)\le \min\{|w|,|y|\}$,
\begin{align*}
I_2&\le c_{15}H(x)e^{c_{16}t_0(|x|)}\int_{\{|w|>\tilde s_1(t)\}}
\frac{e^{-2c_{17}t\min\{g(|w|),g(|y|)\}}H(w)}{t\max\{g(|w|),g(|y|)\}}\left(t^{-d/\alpha}\wedge \frac{t}{|w-y|^{d+\alpha}}\right)\,dw \\
&\le c_{15}e^{c_{16}t_0(|x|)-2c_{17}tg(|x|)}H(x)g(|y|)^{-1}\cdot\int_{\{|w|>\tilde s_1(t)\}}H(w)t^{-1}\left(t^{-d/\alpha}\wedge \frac{t}{|w-y|^{d+\alpha}}\right)\,dw.
\end{align*}
Note that
\begin{align*}
e^{c_{16}t_0(|x|)-2c_{17}tg(|x|)}&=e^{-c_{17}tg(|x|)}e^{c_{16}t_0(|x|)-c_{17}tg(|x|)}.
\end{align*}
Since $\lim_{r\uparrow \infty}g(r)=\infty$, we can find $R_0>0$ large enough so that for all $|x|\ge R_0$,
\begin{align*}
e^{-c_{17}tg(|x|)}\le e^{-2\lambda_1 t},\quad t>0
\end{align*} and
\begin{align*}
 e^{c_{16}t_0(|x|)-c_{17}tg(|x|)}&\le \exp\Big(-t_0(|x|)\left(C_0c_{17}g(|x|)-c_{16}\right)\Big)\le \exp\left(-\frac{C_0c_{17}t_0(|x|)g(|x|)}{2}\right)\\
&\le c_{18},\quad t>C_0t_0(|x|).
\end{align*}
Putting both estimates above together, we obtain
\begin{align*}
e^{c_{16}t_0(|x|)-2c_{17}tg(|x|)}\le c_{19}e^{-2\lambda_1 t},\quad t>C_0t_0(|x|),\ |z|>R_0.
\end{align*}
On the other hand, according to
the procedures for
the estimate of the term  $K_{2}$ in the proof of Lemma \ref{l4-4}, we find that
\begin{align*}
\int_{\{|w|>\tilde s_1(t)\}}H(w)t^{-1}\left(t^{-d/\alpha}\wedge \frac{t}{|w-y|^{d+\alpha}}\right)\,dw\le c_{20}g(|y|)H(y).
\end{align*}
Hence,
\begin{align*}
I_2\le c_{21}e^{-2\lambda_1 t}H(x)H(y),\quad t>2C_0t_0(|x|),\ |x|\le \tilde s_1(t),\ |y|>\tilde s_1(t),\ |x|>R_0.
\end{align*}
Finally, by the estimates of $I_1$ and $I_2$, we obtain
\begin{align*}
p(2t,x,y)\le c_{22}e^{-2\lambda_1 t}H(x)H(y),\quad t>2C_0t_0(|x|),\ |x|\le \tilde s_1(t),\ |y|>\tilde s_1(t),\ |x|>R_0,
\end{align*}
whenever $t>C_0t_0(|x|)$.
This implies that \eqref{l4-6-0} holds for every $|x|\le \tilde s_1(t)$ and $|y|>\tilde s_1(t)$, and the proof of the assertion \eqref{l4-6-0} is finished.
\end{proof}

\begin{lemma}\label{l4-8}
There exist constants $C_0,C_{11}>0$ such that for every $t>C_0t_0(|x|)$,
\begin{equation}\label{l4-8-0}
p(t,x,y)\ge C_{11}e^{-\lambda_1 t} H(x)H(y).
\end{equation}
\end{lemma}
\begin{proof}
Since $t_0(\cdot)$ is almost increasing, we know that $\inf_{x\in \R^d}t_0(|x|)\ge c_{1}>0$. According to \eqref{p3-1-1}, for all $w\in B(0,2)$ and $z\in \R^d$,
\begin{align*}
p(1,z,w)
&\ge c_{2}\left(1 \wedge \frac{1}{|w-z|^{d+\alpha}}\right)\left(\frac{1}{\max\{g(|w|),g(|z|)\}}\wedge 1\right)
e^{-c_{3}\min\{g(|w|),g(|z|)\}}\\
&\ge \frac{c_{4}}{g(|z|)(1+|z|)^{d+\alpha}}.
\end{align*}
Hence, for every $z,u\in \R^d$,
\begin{align*}
p(2,z,u)&=\int_{\R^d}p(1,z,w)p(1,w,u)\,dw\ge \int_{B(0,2)}p(1,z,w)p(1,w,u)\,dw \ge c_5H(z)H(u).
\end{align*}
Therefore,  for every $t>C_0t_0(|x|)>2$ (by taking $C_0$ large enough),
\begin{align*}
p(t,x,y)=&\int_{\R^d}p(t-2,x,z)p(2,z,y)\,dz\ge c_5H(y)\int_{\R^d}p(t-2,x,z)H(z)\,dz\\
=&c_5H(y)T_{t-2}^VH(x)\ge c_6e^{-\lambda_1 t}H(x)H(y),
\end{align*}
where the last inequality follows from \eqref{l4-6-2}. So, the proof of \eqref{l4-8-0} is finished.
\end{proof}

We can now summarise the results in this section as follows.

\begin{proposition}\label{p4-1}
There are constants $C_0, \tilde C_0, C_{12},\ldots,C_{20}>0$ such that the following statements hold.
\begin{itemize}
\item [(1)] If $t_0(\cdot)$ is almost  decreasing,   then, for every
$C_0t_0(|x|)<t\le \tilde C_0$,
\begin{align*}
C_{12}H(x)H(y)\int_{\{|z|\le s_0(t)\}}e^{-C_{13}tg(|z|)}\,dz &\le p(t,x,y)\\
&\le C_{14}H(x)H(y) \int_{\{ |z|\le s_0(C_{15}t)\}} e^{-C_{16}tg(|z|)}\,dz;
\end{align*}
for every $t>\tilde C_0$,
\begin{align*}
 C_{17}e^{-\lambda_{1}t}H(x)H(y) \le p(t,x,y) \le C_{18}e^{-\lambda_{1}t}H(x)H(y).
\end{align*}
\item[(2)] If $t_0(\cdot)$ is almost  increasing,
then for every $t>C_0t_0(|x|)$
\begin{align*}
 C_{19}e^{-\lambda_{1}t}H(x)H(y) \le p(t,x,y) \le C_{20}e^{-\lambda_{1}t}H(x)H(y).
\end{align*}
\end{itemize}
\end{proposition}

\begin{remark}\label{remark4.9}
Note that in Proposition \ref{P:2.4} the time interval can be $0<t\le C_0't_0(|x|)$ with any $C_0'>0$. Therefore, according to \eqref{e4-1},
the estimates in Proposition \ref{p4-1} still hold true for
$C_0t_0(|x|)<t\le \tilde C_0$
by taking a different
constant $C_0$ if necessary.
On the other hand, by the statement (1) in Proposition \ref{p4-1}, the estimates in (1) hold for any $\tilde C_0$ with $\tilde C_0 >C_0t_0(|x|).$
\end{remark}

\section{Justification of Example \ref{ex1}}\label{section5}
By assumption we know that $V$ satisfies \eqref{e1-3} with $g(s)=\log^\beta(e+s)$. So, by \eqref{e1-6}, it is easy to check that $t_0(s)$ satisfies \eqref{e4-1} with $h(s)=\log^{1-\beta}(2+s)$. Then the estimates (1) and (2-i) follow directly from Theorem \ref{th:main}.

It remains to verify (2-ii).
Let $\beta>1$. Since $h(s)=\log^{1-\beta}(2+s)$,
\[
s_0(t):=h^{-1}(t) \vee 2= \begin{cases}
2,\ & t \ge \log^{1-\beta}4,\\
\exp(t^{-1/(\beta-1)})-2,\ & t \in (0,\log^{1-\beta}4].
\end{cases}
\]
For every $0<t\le 2$, first observe that for every positive constants $c_0, c_1$ and $c_2$,
\begin{align*}
 \int_{\{|z|\le c_1  s_0(c_2t)\}}
 e^{-c_0 t
  \log^\beta(2+|z|)
 }\,dz\le c_3(c_1s_0(c_2t))^d\le  c_4\exp(c_5t^{-1/(\beta-1)}).
\end{align*}	
On the other hand, for every $0<t\le2$,
\begin{align*}
\int_{\{|z|\le s_0(t)\}}
 e^{-c_0 t\log^\beta(2+|z|)
 }\,dz\ge \int_{\{|z|\le 2\}}
 e^{-c_0 t\log^\beta(2+|z|)}\,dz\ge c_5.
\end{align*}
Hence,  in order to get the desired lower bound it suffices to consider the time interval
$0<t\le t_0$ for some fixed
$t_0\in (0,\log^{1-\beta}4]$.
Using spherical coordinates and the substitution $s = \log^{\beta}(2+r)$, we get
\begin{align*}
\int_{\{|z|\le s_0(t)\}}
 e^{-c_0 t\log^\beta(2+|z|)
 }\,dz&=c_6\int_0^{s_0(t)}e^{-c_0t \log^\beta(2+r)}r^{d-1}\,dr\\
&=c_6\beta^{-1}\int_{\log^\beta 2}^{t^{-\beta/(\beta-1)}}e^{-c_0 t s}\left(e^{s^{1/\beta}}-2\right)^{d-1}
e^{s^{1/\beta}}s^{-(\beta-1)/\beta}\,ds.
\end{align*}
Choose $\theta\in (0,1/2)$ and $t_0\in (0,\log^{1-\beta}4]$ such that
\begin{align}\label{ex1-1-1}
\theta t_0^{-\beta/(\beta-1)}/2\ge  \log^\beta 2,\quad  d\left(\frac{\theta}{2}\right)^{1/\beta}-c_0\theta\ge c_0\theta/2;
\end{align}
 this is possible because $\beta>1$.
Therefore, for every $0<t\le t_0$,
\begin{align*}
&\int_{\{|z|\le s_0(t)\}}
 e^{-c_0 t\log^\beta(2+|z|)
 }\,dz\\
&\ge c_6\beta^{-1}\int_{\theta t^{-\beta/(\beta-1)}/2}^{\theta t^{-\beta/(\beta-1)}}e^{-c_0 t s}\left(e^{s^{1/\beta}}-2\right)^{d-1}
e^{s^{1/\beta}}s^{-(\beta-1)/\beta}\,ds\\
&\ge c_7t\int_{\theta t^{-\beta/(\beta-1)}/2}^{\theta t^{-\beta/(\beta-1)}}\exp\left(ds^{1/\beta}-c_0 t s\right)\,ds\\
&\ge c_8t^{1-\beta/(\beta-1)}\cdot\exp\left(t^{-1/(\beta-1)}\left(d\left(\frac{\theta}{2}\right)^{1/\beta}-c_0\theta\right)\right)\\
&\ge c_9\exp\left(c_0\theta t^{-1/(\beta-1)}/2\right),
\end{align*}
where the last inequality follows from \eqref{ex1-1-1}. Thus, the proof is complete.

\section{Proofs of main results}\label{section6}
\begin{proof}[Proof of Theorem $\ref{th:main}$]
The conclusion is a direct consequence of Proposition \ref{P:2.4} and Proposition \ref{p4-1}.
See Remark \ref{remark4.9} for the additional comments.
\end{proof}

\begin{proof}[Proof of Theorem $\ref{prop:Green}$]
Let $x,y\in \R^d$ with
$|x|\le |y|$
and $x \neq y$.  We write
\[
G^V(x,y) = \int_0^{C_0 t_0(|x|)} p(t,x,y)\, dt + \int_{C_0 t_0(|x|)}^\infty p(t,x,y)\, dt =: I + II,
\]
 where the constant $C_0>0$ comes from Theorem \ref{th:main}.

\ \

We first apply the estimates from Theorem \ref{th:main}(1) to estimate the integral $I$.  An analysis is divided into three cases.

\smallskip

\noindent
\textbf{Case 1}:
$2|x-y|^\alpha\le C_0t_0(|x|)$ and $2/g(|x|)<|x-y|^\alpha$.

We write
\[
I = \left(\int_0^{1/g(|y|)} + \int_{1/g(|y|)}^{|x-y|^{\alpha}} + \int_{|x-y|^{\alpha}}^{C_0t_0(|x|)}  \right) p(t,x,y)\, dt =: I_1 + I_2 + I_3.
\]
According to \eqref{t1-1-1}, we have
\[
I_1 \asymp |x-y|^{-d-\alpha} \int_0^{1/g(|y|)} t e^{-c t g(|x|)} \,dt \asymp |x-y|^{-d-\alpha} \int_0^{1/g(|y|)} t  \,dt \asymp |x-y|^{-d-\alpha} g(|y|)^{-2}.
\]
We point out that all of the comparability formulas in this proof contain the exponential terms with a generic constant $c$, but this is only a shorthand notation. The exponential term $e^{-c t g(|x|)}$  comes from the two-sided bound in Theorem \ref{th:main}(1) and the reader should note that we always have $c=C_2$ in lower estimates and $c=C_4$ in upper estimates, for different $C_2, C_4$.

Similarly, by \eqref{t1-1-1},
\begin{align*}
I_2 & \asymp  |x-y|^{-d-\alpha} g(|y|)^{-1} \int_{1/g(|y|)}^{|x-y|^{\alpha}} e^{-c t g(|x|)}\, dt \\
&\asymp |x-y|^{-d-\alpha} g(|x|)^{-1}g(|y|)^{-1}  \int_{g(|x|)/g(|y|)}^{|x-y|^\alpha g(|x|)} e^{-c s} \,ds \\
&\asymp 		|x-y|^{-d-\alpha} g(|x|)^{-1}g(|y|)^{-1},
\end{align*}
where in the third step we have used the fact that (since $g(|x|)/g(|y|)\le 1$ and $|x-y|^\alpha g(|x|)\ge 2$)
\begin{align*}
c_1\le \int_1 ^2 e^{-cs}\,ds\le \int_{g(|x|)/g(|y|)}^{|x-y|^\alpha g(|x|)} e^{-c s} \,ds\le \int_0^\infty e^{-cs}\,ds\le c_2.
\end{align*}
Moreover, by applying \eqref{t1-1-1}  and using
substitution $t = |x-y|^{\alpha} s$, we get
\begin{align*}
I_3 & \asymp g(|y|)^{-1} \int_{|x-y|^{\alpha}}^{C_0t_0(|x|)} t^{-d/\alpha-1} e^{-ct g(|x|)} \,dt \\
    & \asymp |x-y|^{-d} g(|y|)^{-1} \int_1^{C_0t_0(|x|)/|x-y|^{\alpha}} s^{-d/\alpha-1} e^{-c|x-y|^{\alpha}g(|x|) s }\, ds \\
		& \asymp |x-y|^{-d} g(|y|)^{-1} e^{-c|x-y|^{\alpha}g(|x|)},
\end{align*}
where the third step follows from the estimates
\begin{align*}
\int_1^{C_0t_0(|x|)/|x-y|^{\alpha}} s^{-d/\alpha-1} e^{-c|x-y|^{\alpha}g(|x|) s }\, ds
      &\leq e^{-c |x-y|^{\alpha}g(|x|)} \int_1^{\infty} s^{-d/\alpha-1} \,ds\\
      &\le c_3e^{-c |x-y|^{\alpha}g(|x|)}
\end{align*}
and
\begin{align*}
\int_1^{C_0t_0(|x|)/|x-y|^{\alpha}} s^{-d/\alpha-1} e^{-c |x-y|^{\alpha}g(|x|) s } \,ds
      &\geq e^{-2 c |x-y|^{\alpha}g(|x|)} \int_1^{2} s^{-d/\alpha-1} \,ds\\
      &\ge c_4e^{-2c |x-y|^{\alpha}g(|x|)}.
\end{align*}
In particular,
\begin{align*}
I_3 &\leq c_5 |x-y|^{-d-\alpha} g(|x|)^{-1} g(|y|)^{-1}.
\end{align*}
Consequently, combining all the estimates above with the facts $g(|y|)\ge g(|x|)$ and $Q(x,y)\asymp 1$ (since $2/g(|x|)<|x-y|^\alpha$), we have
\[
I \asymp |x-y|^{-d-\alpha}g(|x|)^{-1}g(|y|)^{-1}\asymp |x-y|^{-d+\alpha}\cdot\left(\frac{1}{g(|y|)|x-y|^\alpha}\wedge 1\right)
\left(\frac{1}{g(|x|)|x-y|^\alpha}\wedge 1\right)Q(x,y).
\]

\smallskip

\noindent

\textbf{Case 2}: $2|x-y|^{\alpha}  \leq C_0t_0(|x|)$ and $2/g(|x|) \geq |x-y|^{\alpha}$.
In this case, we will use the simpler decomposition
\[
I = \left(\int_0^{|x-y|^{\alpha}/2} + \int_{|x-y|^{\alpha}/2}^{C_0t_0(|x|)}  \right) p(t,x,y)\, dt =: J_1 + J_2.
\]
Indeed, $|x-y|\le \left(2/g(|x|)\right)^{1/\alpha}\le c_6$, so it holds that $g(|x|)\asymp g(|y|)$ for this range of $x$ and $y$, which implies that
\begin{equation}\label{t1-2-1}
(t g(|y|))^{-1} \wedge 1) e^{-c t g(|x|)}  \asymp e^{- t g(|x|)},\quad 0\le t \le C_0t_0(|x|).
\end{equation}
With this and \eqref{t1-1-1},
\[
J_1 \asymp |x-y|^{-d-\alpha} \int_0^{|x-y|^{\alpha}/2} t e^{-c t g(|x|)} \,dt
\asymp |x-y|^{-d-\alpha} \int_0^{|x-y|^{\alpha}/2} t\, dt
\asymp |x-y|^{-d+\alpha},
\]
where we have used the fact that $|x-y|^\alpha g(|x|)\le 2$.

Meanwhile, by \eqref{t1-1-1}, \eqref{t1-2-1} and substitution $t = |x-y|^{\alpha} s/2$, we derive
\begin{align*}
J_2 &\asymp \int_{|x-y|^{\alpha}/2}^{C_0t_0(|x|)} t^{-d/\alpha} e^{-c t g(|x|)} dt
    \asymp |x-y|^{-d+\alpha} \int_{1}^{\frac{2C_0t_0(|x|)}{|x-y|^{\alpha}}} s^{-d/\alpha} e^{-c s g(|x|)|x-y|^{\alpha}} ds.
\end{align*}

If $\alpha < d$, then, by the fact $2|x-y|^{\alpha}  \leq C_0t_0(|x|)$ and $|x-y|^\alpha g(|x|)\le 2$, it holds that
\begin{align*}
c_7\le c_8\int_1^2 s^{-d/\alpha}\,ds \le
c_9
\int_{1}^{\frac{2C_0t_0(|x|)}{|x-y|^{\alpha}}} s^{-d/\alpha} e^{-c s g(|x|)|x-y|^{\alpha}} \,ds\le c_{10}\int_1^\infty s^{-d/\alpha}\,ds=\frac{ c_{10}\alpha}{d-\alpha}.
\end{align*}
If $\alpha=d=1$, then, by noting from \eqref{e1-6} and
\eqref{l4-2-1c}
that $2C_0t_0(|x|)\ge c_{11}g(|x|)^{-1}$,
\begin{align*}
\int_{1}^{\frac{2C_0t_0(|x|)}{|x-y|^{\alpha}}} s^{-d/\alpha} e^{-c s g(|x|)|x-y|^{\alpha}} \,ds&\ge c_{12}\int_{1}^{
\max\left\{\frac{c_{11}}{g(|x|)|x-y|^\alpha},4\right\}} s^{-1}\,ds\\
&
\ge c_{13}\max \left\{\log \left(\frac{1}{|x-y|^\alpha g(|x|)}\right),1\right\}
\end{align*} and
\begin{align*}
&\int_{1}^{\frac{2C_0t_0(|x|)}{|x-y|^{\alpha}}} s^{-d/\alpha} e^{-c s g(|x|)|x-y|^{\alpha}} \,ds\\
&\le
c_{14}\int_{1}^{\max\left\{\frac{c_{11}}{g(|x|)|x-y|^\alpha},1\right\}} s^{-1}\,ds+c_{14}g(|x|)|x-y|^{\alpha}\cdot\int_{\max\left\{\frac{c_{11}}{g(|x|)|x-y|^\alpha},1\right\}}^\infty
e^{-c s g(|x|)|x-y|^{\alpha}}\, ds\\
&\le c_{15}\max \left\{\log \left(\frac{c_{11}}{g(|x|)|x-y|^\alpha}\right),1\right\}.
\end{align*}
If $\alpha > d=1$, then
\begin{align*}
\int_{1}^{\frac{2C_0t_0(|x|)}{|x-y|^{\alpha}}} s^{-d/\alpha} e^{-c s g(|x|)|x-y|^{\alpha}} \,ds&\ge c_{16}\int_{1}^{
\max\left\{\frac{c_{11}}{g(|x|)|x-y|^\alpha},4\right\}} s^{-d/\alpha}\,ds\\
&\ge c_{17}\left(\frac{1}{|x-y|^\alpha g(|x|)}\right)^{(\alpha-d)/\alpha}
\end{align*} and
\begin{align*}
&\int_{1}^{\frac{2C_0t_0(|x|)}{|x-y|^{\alpha}}} s^{-d/\alpha} e^{-c s g(|x|)|x-y|^{\alpha}} \,ds\\
&\le
c_{18}\int_{1}^{
\max\left\{\frac{c_{11}}{g(|x|)|x-y|^\alpha},1\right\}} s^{-d/\alpha}\,ds\\
&\quad+c_{18}\left(g(|x|)|x-y|^{\alpha}\right)^{-d/\alpha}\cdot\int_{
\max\left\{\frac{c_{11}}{g(|x|)|x-y|^\alpha},1\right\}}^\infty
e^{-c s g(|x|)|x-y|^{\alpha}}\, ds\\
&\le c_{19}\left(\frac{1}{|x-y|^\alpha g(|x|)}\right)^{(\alpha-d)/\alpha}.
\end{align*}
Hence, putting all above estimates together, we obtain
\begin{align*}
J_2\asymp |x-y|^{-d+\alpha}Q(x,y),
\end{align*}
which implies immediately that
\begin{align*}
I \asymp |x-y|^{-d+\alpha}Q(x,y)\asymp |x-y|^{-d+\alpha}\cdot\left(\frac{1}{g(|y|)|x-y|^\alpha}\wedge 1\right)
\left(\frac{1}{g(|x|)|x-y|^\alpha}\wedge 1\right)Q(x,y),
\end{align*}
where we have used the fact that $|x-y|^\alpha\le 2/g(|x|)$ and $g(|x|)\asymp g(|y|)$.

\smallskip

\noindent
\textbf{Case 3}:
$2|x-y|^\alpha>C_0t_0(|x|)$.

Similarly, it follows from \eqref{e1-6} and
\eqref{l4-2-1c}
that $C_0t_0(|x|)>2c_{20}g(|x|)^{-1}\ge 2c_{20}g(|y|)^{-1}$. Let
\begin{align*}
I = \left(\int_0^{c_{20}/g(|y|)} + \int_{c_{20}/g(|y|)}^{C_0t_0(|x|)}  \right) p(t,x,y) \,dt =: K_1 + K_2.
\end{align*}
So by \eqref{t1-1-1}, we get
\begin{align*}
K_1 &\asymp |x-y|^{-d-\alpha}\int_0^{c_{20}/g(|y|)}t \,dt\asymp |x-y|^{-d-\alpha}g(|y|)^{-2}
\end{align*} and
\begin{align*}
K_2 & \asymp  |x-y|^{-d-\alpha} g(|y|)^{-1} \int_{c_{20}/g(|y|)}^{C_0t_0(|x|)} e^{-c t g(|x|)}\, dt  \\
    & \asymp  |x-y|^{-d-\alpha} g(|x|)^{-1} g(|y|)^{-1},
\end{align*}
where in the last step we have used the fact that
\begin{align*}
c_{21}g(|x|)^{-1}\le & \int_{c_{20}/g(|y|)}^{2c_{20}/g(x)} e^{-c t g(|x|)} \,dt\le \int_{c_{20}/g(|y|)}^{C_0t_0(|x|)} e^{-c t g(|x|)} \,dt
\\ \le &\int_{c_{20}/g(|y|)}^{\infty} e^{-c t g(|x|)} \,dt\le c_{22}g(|x|)^{-1}.
\end{align*}

Putting all the estimates above together and combining them with the facts that $g(|y|)\ge g(|x|)$ and $Q(x,y)\asymp 1$ (since $c_{20}/g(|x|)<|x-y|^\alpha$),  we have
\[
I \asymp |x-y|^{-d-\alpha}g(|x|)^{-1}g(|y|)^{-1}\asymp |x-y|^{-d+\alpha}\cdot\left(\frac{1}{g(|y|)|x-y|^\alpha}\wedge 1\right)
\left(\frac{1}{g(|x|)|x-y|^\alpha}\wedge 1\right)Q(x,y).
\]

By now we have obtained that for all of these cases,
\[
I \asymp \frac{1}{|x-y|^{d-\alpha}} \left(1 \wedge \frac{1}{g(|x|)|x-y|^{\alpha}}\right) \left(1 \wedge \frac{1}{g(|y|)|x-y|^{\alpha}}\right) Q(x,y).
\]

\ \

Next, we are going to estimate the integral $II$. We will show that there is a constant $c_{23}>0$ such that $II \le c_{23} I$, which will lead us immediately to the
desired conclusion. Recall we assume that $|x| \leq |y|$.

When $t_0(\cdot)$ is almost increasing,  it follows directly from  \eqref{t1-1-2} that
\begin{equation}\label{t1-2-2}
II \le  \frac{c_{24}}{g(|x|)g(|y|) (1+|x|)^{d+\alpha}(1+|y|)^{d+\alpha}}.
\end{equation}
Suppose that $t_0(\cdot)$ is almost  decreasing.
Then we set
\begin{align*}
II= \left(\int_{C_0t_0(|x|)}^{\tilde C_0} + \int_{\tilde C_0}^{\infty} \right) p(t,x,y) \,dt =: L_1 + L_2,
\end{align*}
where $\tilde C_0$ is the same constant in Theorem \ref{th:main}.
By \eqref{t1-1-4} it is direct to see that
\begin{align*}
L_2\le  \frac{c_{25}}{g(|x|)g(|y|) (1+|x|)^{d+\alpha}(1+|y|)^{d+\alpha}}.
\end{align*}
By tracking the proof of Theorem
\ref{th:main}(2-ii) we can take $C_0>C_*C_{10}^{-1}$, where $C_*$, $C_{10}$ are positive constants in
\eqref{e4-1} and \eqref{t1-1-3}, respectively. Hence, according to \eqref{e4-1},
for every
$C_0t_0(|x|)\le t \le \tilde C_0$,
\begin{align*}
s_0(C_{10}t)&\le s_0\left(C_0C_{10}t_0(|x|)\right)\le s_0\left(C_0C_{10}C_*^{-1}h(|x|)\right)
\le s_0\left(h(|x|)\right)\le (|x|\vee 2)\le 2(1+|x|).
\end{align*}
Then, by \eqref{t1-1-3},
\begin{align*}
L_1&\le  c_{26}H(x)H(y)
\int_{C_0t_0(|x|)}^{\tilde C_0}
\int_{\{|z|\le c_{27}s_0(C_{10}t)\}} e^{-c_{28} tg(|z|)}\,dz\,dt \\
& \le c_{29}H(x)H(y)(1+|x|)^d.
\end{align*}
Consequently, combining the estimates of $L_1$ and $L_2$ with \eqref{t1-2-2}, we can conclude that if $t_0(\cdot)$ is an almost monotone function
and $C_0t_0(|x|) \le h(2)$,
then
\[
II \le
\frac{c_{30}}{g(|x|)g(|y|) (1+|y|)^{d+\alpha}}.
\]
On the other hand, when $C_0t_0(|x|) > h(2)$, $II \leq L_2$ and so it is clear that the inequality above holds true.

Observe now that for the {\bf Case 1} and {\bf Case 3} above, it holds that $|x-y|^{\alpha}>c_{31}g(|x|)^{-1}\ge c_{31}g(|y|)^{-1}$. Then, we can write
\begin{align*}
II & \le  \frac{c_{30}}{g(|x|)g(|y|) (1+|y|)^{d+\alpha}} \\
   & = \frac{c_{30}  |x-y|^{d+\alpha}}{|x-y|^{d-\alpha}g(|x|)|x-y|^{\alpha}g(|y|)|x-y|^{\alpha} (1+|y|)^{d+\alpha}} \\
	   & \le  \frac{c_{32}}{|x-y|^{d-\alpha}} \left(1 \wedge \frac{1}{g(|x|)|x-y|^{\alpha}}\right) \left(1 \wedge \frac{1}{g(|y|)|x-y|^{\alpha}}\right)\frac{|x-y|^{d+\alpha}}{(1+|y|)^{d+\alpha}} \\
	 & \le \frac{c_{32}}{|x-y|^{d-\alpha}} \left(1 \wedge \frac{1}{g(|x|)|x-y|^{\alpha}}\right) \left(1 \wedge \frac{1}{g(|y|)|x-y|^{\alpha}}\right)  \frac{(2|y|)^{d+\alpha}}{(1+|y|)^{d+\alpha}}  \\
	&  \le c_{33} I.
\end{align*}
For the remaining {\bf Case 2} above, it holds that $|x-y|^{\alpha}\le c_{34}g(|x|)^{-1}\le c_{35}$
and $g(|x|)\asymp g(|y|)$.
By this we obtain immediately that
\begin{align*}
II\le c_{36}\le c_{37}I.
\end{align*}
So now we have finished the proof.
\end{proof}

 \bigskip

\noindent {\bf Acknowledgements.}\,\, The research of Xin Chen is supported by
the National Natural Science Foundation of China (No.\ 12122111).
The research of Kamil Kaleta is supported by the National Science Centre, Poland, project OPUS-18 2019/35/B/ST1/02421.
The research of Jian Wang is supported by the National Key R\&D Program of China (2022YFA1006003) and  the National Natural Science Foundation of China (No.\ 12225104).








\noindent {\bf Conflict of Interest}\,\,
The authors have no competing interests to declare that are relevant to the content of this
article.



\end{document}